\documentclass{amsart}
\usepackage{amsmath}
\usepackage{amssymb}
\usepackage{amsthm}
\usepackage[all]{xy}
\usepackage{multicol}
 \usepackage{color}
\usepackage{graphicx}
\usepackage{tikz}

\usepackage{stmaryrd} 

\numberwithin{equation}{section}

\newtheorem{theorem}{Theorem}[section]
\newtheorem{definition}[theorem]{Definition}

\newtheorem{example}[theorem]{Example}
\newtheorem{lemma}[theorem]{Lemma}
\newtheorem{corollary}[theorem]{Corollary}
\newtheorem{remark}[theorem]{Remark}
\newtheorem{notation}[theorem]{Notation}

\newcommand{\restr}{\restriction_}
\newcommand{\loc}{\mathrm{loc}}
\newcommand{\Id}{\mathbb{I}}

\newcommand{\Ll}{{\mathcal L}}
\newcommand{\R}{{\mathbb R}}
\newcommand{\re}{{\mathbb R}}
\newcommand{\N}{{\mathbb N}}
\newcommand{\C}{{\mathbb C}}
\newcommand{\W}{{\mathbb W}}
\newcommand{\h}{{\mathbb H}}

\newcommand{\Hn}{{\mathbb H^n}}

\newcommand{\norm}[1]{\|{#1}\|_{\infty}}
\newcommand{\HH}{{\mbox{H}\Hn}}
\renewcommand{\t}{{\tau}}
\newcommand{\om}{{\omega}}

\newcommand{\F}{{\Phi}}
\newcommand{\de}{{\delta}}
\newcommand{\f}{{\phi}}
\newcommand{\p}{{\psi}}
\newcommand{\Wf}{W^\phi}

\newcommand{\CH}{\mathrm C^1_{\h}}

\newcommand{\gf}{\nabla^\phi}
\newcommand{\dede}[1]{{\frac{\partial}{\partial #1}}}

\newcommand{\df}{{d_\phi}}
\newcommand{\mfrl}{\mathfrak L^{\infty}} 
\newcommand{\rmLinf}{\mathrm L^\infty}

\DeclareMathOperator*{\clos}{clos}
\DeclareMathOperator*{\ri}{r.i.}

\DeclareMathOperator{\sgn}{sgn}

\allowdisplaybreaks
\begin{document}

\title[Intrinsic Lipschitz graphs in $\mathbb H^{n}$]{Intrinsic Lipschitz graphs in Heisenberg groups and continuous solutions of a balance equation}


\author[F. Bigolin, L. Caravenna, F. Serra Cassano]{Francesco Bigolin, Laura Caravenna, Francesco Serra Cassano}
\address{Francesco Bigolin, Dipartimento di Matematica di Trento,
via Sommarive 14, 38050 Povo (Trento), Italy;}
\email{bigolin@science.unitn.it}
\address{Laura Caravenna, Oxford Centre for Nonlinear PDE, Mathematical Institute, University of Oxford, 24-29 St Giles', Oxford OX1 3LB, U.K.;}
\email{laura.caravenna@maths.ox.ac.uk}
\address{Francesco Serra Cassano, Dipartimento di Matematica di Trento,
via Sommarive 14, 38050 Povo (Trento), Italy;}
\email{bigolin@science.unitn.it}

\date{\today}

\thanks{F.B. is supported by PRIN 2008 and University of Trento, Italy.  L.C. has been supported by Centro De Giorgi (SNS di Pisa), by the ERC Starting Grant CONSLAW and the UK EPSRC Science and Innovation award to the Oxford Centre for Nonlinear PDE (EP/E035027/1). F.S.C. is supported by GNAMPA, PRIN 2008 and University of Trento, Italy.}


\begin{abstract}
 
In this paper we provide a characterization of intrinsic Lipschitz graphs in the sub-Riemannian Heisenberg groups in terms of their distributional gradients. Moreover, we prove the equivalence of different notions of continuous weak solutions to the equation $\phi_{y}+ [\phi^{2}/2]_{t}=w$, where $w$ is a bounded measurable function depending on $\phi$. 
\end{abstract}

\maketitle

\tableofcontents

\section{Introduction}


In the last years it has been largely developed the study of intrinsic submanifolds inside the Heisenberg groups $\Hn$ or more general Carnot groups, endowed with their Carnot-Carath\'eodory metric structure, also named sub-Riemannian. 
By an intrinsic regular (or intrinsic Lipschitz) hypersurfaces we mean a submanifold which has, in the intrinsic geometry of $\Hn$, the same role like a $\mathrm C^1$ (or Lipschitz) regular graph has in the Euclidean geometry. Intrinsic regular graphs had several applications within the theory of rectifiable sets and minimal surfaces in CC geometry, in theoretical computer science, geometry of Banach spaces and mathematical models in neurosciences, see \cite{cdpt}, \cite{CK} and the references therein.

We postpone complete definitions of $\Hn$ to Section \ref{S:hnrecalls}. We only remind that the Heisenberg group $\Hn = \C^n \times \R \equiv \R^{2n+1}$ is the simplest example of Carnot group, endowed with a left-invariant metric $d_{\infty}$ (equivalent to its Carnot-Carath\'eodory metric), not equivalent to the Euclidean metric. $\Hn$ is a (connected, simply connected and stratified) Lie group and has a sufficiently rich compatible underlying structure, due to the existence of intrinsic families of left translations and dilations and depending on the horizontal vector fields $X_1,...,X_{2n}$.
We call intrinsic any notion depending directly by the structure and geometry of $\Hn$. For a complete description of Carnot groups \cite{cdpt,gromov2,montg,pansu,pansuthese} are recommended.

As we said, we will study intrinsic submanifolds in $\Hn$. An intrinsic regular hypersurface $S\subset\Hn$ is locally defined as the non critical level set of an horizontal differentiable function, more precisely there exists locally a continuous function $f:\Hn\to\R$ such that $S=\{f=0\}$ and the intrinsic gradient $\nabla_{\h}f=(X_1f,\dots,X_{2n}f)$ exists in the sense of distributions and it is continuous and non-vanishing on $S$ .
Intrinsic regular hypersurfaces can be locally represented as $X_1$-graph by a function $\f:\om\subset\W\equiv\R^{2n}\to\R$, where $\W=\{x_1=0\}$, through an implicit function theorem (see \cite{frasersercas}). In \cite{ascv,BSC1,BSC2} the parametrization $\f$ has been characterized as weak solution of a system of non linear first order PDEs $\gf\f=w$, where $w:\om\to\R^{2n-1}$ and $\gf=(X_2,\dots,X_n,\partial_y+\f\partial_t,X_{n+2},\dots,X_{2n} )$, (see Theorem \ref{miotheorem}). By an intrinsic point of view, the operator $\gf\f$ acts as the intrinsic gradient of the function $\f:\W\to\R$. In particular it can be proved that $\f$ is a continuous distributional solution of the problem $\gf\f=w$ with $w\in \mathrm C^0(\om,\R^{2n-1})$ if and only if $\f$ induces an intrinsic regular graph, (see \cite{BSC2}).

Let us point out that an intrinsic regular graph can be very irregular from the Euclidean point of view: indeed, there are examples of intrinsic regular graphs in $\h^{1}$ which are fractal sets in the Euclidean sense (\cite{kirsercas}).

The aim of our work is to characterize intrinsic Lipschitz graphs in terms of the intrinsic distributional gradient. It is well-know that in the Euclidean setting a Lipschitz graph $S=\{(A,\phi(A)):\ A\in\omega\}$, with $\phi:\omega\subset\R^{m}\to\R$ can be equivalently defined 
\begin{itemize}
\item by means of cones: there exists $L>0$ such that $$C\left((A_0,\phi(A_0));\tfrac{1}{L}\right)\cap S=\{(A_0,\phi(A_0))\}$$ for each $A_0\in\omega$, where
$C\left((A_0,\phi(A_0));\frac{1}{L}\right)=\{(A,s)\in\R^m\times\R:\ |A-A_0|_{\R^m}\leq\frac{1}{L}|s-\phi(A_0)|$;
\item in a metric way: there exists $L>0$ such that $|\f(A)-\f(B)|\leq L|A-B|$ for every $A,B\in\om$;
\item  by the distributional derivatives: there exist the distributional derivatives $$\frac{\partial \f}{\partial x_i}\in \mathrm L^{\infty}(\om)\qquad \forall i=1,\dots,m$$
provided that $\om$ is a regular connected open bounded set.
\end{itemize}
Intrinsic Lipschitz graphs in $\Hn$ have been introduced in \cite{frasersercaslip}, by means of a suitable notion of intrinsic cone in $\Hn$. As consequence, the metric definition (see Definition \ref{D:deflip}) is given with respect to the the graph quasidistance $d_\f$, (see \eqref{fidistanzadef2708}) i.e the function $\f:(\om,d_\f)\to\R$ is meant Lipschitz in classical metric sense.
This notion turns out to be the right one in the setting of the intrinsic rectifiability in $\Hn$. Indeed, for instance, it was proved in \cite{frasersercaslip} that the notion of rectifiable set in terms of an intrinsic regular hypersurfaces is equivalent to the one in terms of intrinsic Lipschitz graphs.

We will denote by $\mathrm{ Lip}_\W(\om)$ the class of all intrinsic Lipschitz function $\f:\om\to\R$ and by $\mathrm{ Lip}_{\W,\loc}(\om)$ the one of locally intrinsic Lipschitz functions. Notice that $\mathrm{ Lip}_\W(\om)$ is not a vector space and that
$$\mathrm{ Lip}(\om)\subsetneq \mathrm{ Lip}_{\W,\loc}(\om)\subsetneq \mathrm C^{0,1/2}_{\loc}(\om),$$
where $\mathrm{ Lip}(\om)$ and $\mathrm C^{0,1/2}_{\loc}(\om)$ denote respectively the classes of Euclidean Lipschitz and $1/2$-H\"older functions in $\om$. For a complete presentation of intrinsic Lipschitz graphs \cite{pinamonti,frasersercaslip} are recommended.

The first main result of this paper is 
the characterization of a parametrization $\f:\om\to\R$ of an intrinsic Lipschitz graph as a continuous distributional solution of $\gf\f=w$, where $w\in \mathrm L^{\infty} (\om,\R^{2n-1})$.

\begin{theorem}
\label{T:firsttheorem}
Let $\omega\subset \W\equiv\R^{2n}$ be an open set, $\f:\omega\rightarrow\R$ be a continuous function and $w\in \mathrm L^{\infty}(\omega;\R)$. 
$\f\in \mathrm{ Lip}_{\W,\loc}(\om;\R)$ if and only if there exists $w\in \rmLinf_{loc}(\omega;\R^{2n-1})$ such that $\f$ is a distributional solution of the system $\gf\f=w$ in $\om$.
\end{theorem}
%
%
%

We stress that this is indeed different from proving a Rademacher theorem, which is more related to a pointwise rather than distributional characterization for the derivative, see \cite{frasersercaslip}.
Nevertheless, we find that the density of the (intrinsic) distributional derivative is indeed given by the function one finds by Rademacher theorem.
We also stress that there are a priori different notions of \emph{continuous} solutions $\phi:\omega\to\R$ to $\gf\f=w$, which express the Lagrangian and Eulerian viewpoints.
They will turn out to be equivalent descriptions of intrinsic Lipschitz graphs, when the source $w$ belongs to $\rmLinf(\omega;\R^{2n-1})$.
This is proved in Section \ref{S:furthereq} and it is summarized as follows.
\begin{theorem}
\label{T:othertheorem} Let $\phi:\om\to\R$ be a continuous function.
The following conditions are equivalent
\begin{enumerate}
\item \label{item:distr} $\f$ is a distributional solution of the system $\gf\f=w$ with $w\in\rmLinf(\omega;\R^{2n-1})$;
\item \label{item:lagr} $\f$ is a broad solution of $\gf\f=w$, i.e.~there exists a Borel function $\hat w\in \mfrl(\omega;\R^{2n-1})$ s.t.
\begin{itemize}
	\item[(B.1):] $w(A)=\hat w(A)$ $\mathcal L^{2n}$-a.e. $A\in\om$;
	\item[(B.2):] for every continuous vector field $\nabla^{\phi}_{i} $ ($i=2,\dots,2n$) having an integral curve $\Gamma\in \mathrm C^{1}((-\de,\de); \omega)$, $\phi\circ\Gamma$ is absolutely continuous and at points $s$ of differentiability it satisfies
	\[ \displaystyle{\frac{d}{ds}\phi\left(\Gamma(s)\right)= \hat w_i\left(\Gamma(s)\right)}.\]
\end{itemize}
\end{enumerate}
\end{theorem}
In the statement, $\mfrl(\omega;\R^{2n-1})$ denotes the set of measurable bounded functions from $\omega$ to $\R^{2n-1}$, while $\rmLinf(\omega;\R^{2n-1})$ denotes the equivalence classes of Lebesgue measurable functions in $\mfrl(\omega;\R^{2n-1})$ which are identified when differing on a Lebesgue negligible set. We will keep this notation throughout the paper: its relevance is remarked by Examples \ref{Ex:1}, \ref{Ex:2} below.
\paragraph{\bf Outline of the proofs} With the intention of focussing on the nonlinear field, we fix the attention on the case $n=1$. The variables will be denoted by $t$ and $y$, and the subscripts $[\cdot]_{t}$, $[\cdot]_{z}$ will denote the distributional derivatives $\dede{t}=\partial_{t}$, $\dede{z}=\partial_{z}$ in the Euclidean sense w.r.t.~these variables.

Given a continuous distributional solution $\phi\in C^0(\om)$ of the PDE
\[
\gf\f (z,t)= \f_{z}(z,t)+\left[ \frac{\f^{2}(z,t)}{2}\right]_{t}=w(z,t)\qquad {\rm for}\ (z,t)\in \om,
\] 
we first prove that it is Lipschitz when restricted along any characteristic curve $\Gamma(z)=(z,\gamma(z))$, where $\dot\gamma=\phi\circ\Gamma$. The proof follows a previous argument by Dafermos (see Lemma \ref{L:DafLipschitz}). By a construction based on the classical existence theory of ODEs with continuous coefficients, we can then define a change of variable $(z,\chi(z,\tau))$ which straightens characteristics. This change of variables does not enjoy $\mathrm{BV}$ or Lipschitz regularity, it fails injectivity in an essential way, though it is continuous and we impose an important monotonicity property. This monotonicity, relying on the fact that we basically work in dimension $2$, is the regularity property which allows us the change of variables. As we exemplify below, we indeed have an approach different from providing a regular Lagrangian flow of Ambrosio-Di Perna-Lions's theory, and it is essentially two dimensional.
After the change of variables, the PDE is, roughly, linear, and we indeed find a family of ODEs for $\phi$ on the family of characteristics composing $\chi$, with coefficients which now are not anymore continuous, but which are however bounded.
By generalizing a lemma on ODEs already present in~\cite{BSC1}, we prove the $1/2$-H\"older continuity of $\f$ on the vertical direction ($z$ constant), and a posteriori in the whole domain.
This are the main ingredients for establishing that $\phi$ defines indeed a Lipschitz graph: given two points, we connect them by a curve made first by a characteristic curve which joins the two vertical lines through the points, then by the remaining vertical segment.
We manage this way to control the variation of $\phi$ between the two points with their graph distance $\df$, checking therefore the metric definition of intrinsic Lipschitz graphs.

The other implication of Theorem~\ref{T:firsttheorem} is based on the possibility of suitably approximating an intrinsic Lipschitz graph with intrinsic regular graphs. A geometric approximation is provided by~\cite{pinamonti}. We also provide a more analytic, and weaker, approximation as a byproduct of the change of variable $\chi$ which straightens characteristics, by mollification (see the proof of Theorem~\ref{T:converseDafermos}).

We stop now for a while in order to clarify the features of the statement in Theorem~\ref{T:othertheorem}, and why it is so important to differentiate between  $\mfrl(\omega;\R^{2n-1})$ and $\rmLinf(\omega;\R^{2n-1})$.
Lagrangian formulations are affected by altering the representative, as the following example stresses.

\begin{example}[Figure~\ref{fig:1}]
\label{Ex:1}
{Let $\om=(0,1)\times(-1,1)$ and let $\f,w:\om\to\R$ be the functions defined as
$$\f(z,t):=\sqrt{|t|},\qquad w(z,t):=\left\{\begin{array}{ll}
1/2 & {\rm if}\,t\geq 0 \\
-1/2 & {\rm if}\,t < 0
\end{array}\right.$$ 
Then it is easy to verify that $\f$ is a continuous distributional solution of 
$$\gf\f=\dede{z}\f+\dede{t}\frac{\f^2}{2}=w.$$
Consider the specific characteristic curve $(z,\gamma(z)):=(z,0)$. Even if $\dot\gamma(z)=\phi(z,\gamma(z))=0$, the derivative of $\phi$ along this characteristic curve is not the right one:
$$\dede{z}\f(z,0)=0\neq\frac{1}{2}=w(z,0).$$
Equation~\eqref{E:evolutionphi} holds however on every characteristic curves provided we choose correctly an $\mathrm L^{\infty}$-representative $\hat w\in\mfrl(\omega)$ of the source $w$: it is enough to consider 
$$\hat w(z,t):=\left\{\begin{array}{ll}
1/2 & {\rm if}\,t>0 \\
0 & {\rm if}\, t=0\\
-1/2 & {\rm if}\,t< 0
\end{array}\right.$$
Notice that $w(z,t)=\hat w(z,t)$ for $\mathcal L^{2}$-a.e. $(z,t)\in\om$.$\hfill\square$
}
\end{example}

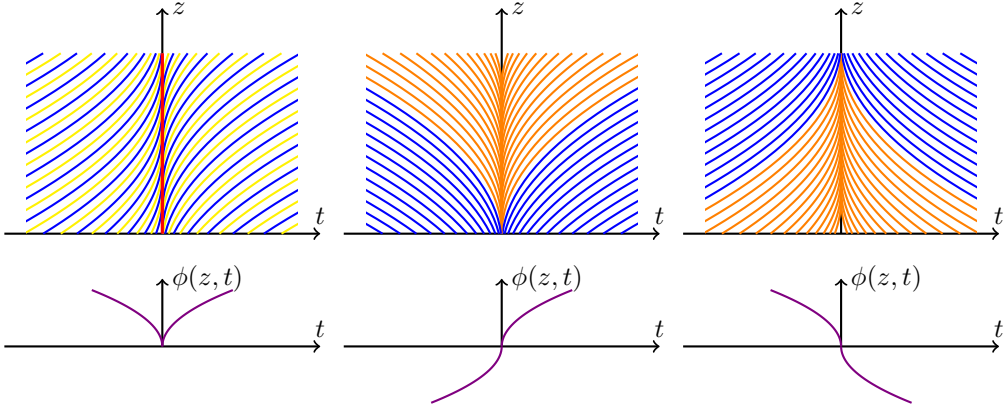
\begin{figure}
\centering
\begin{tikzpicture}[x=.6cm, y=.6cm, style= thick]
\draw[->,color=black] (-1.5,0) -- (5.5,0) node[above] {$t$};
\draw[->,color=black] (2,0) -- (2,5) node[right] {$z$};
\begin{scope}
\clip (-1, -1) rectangle (5, 5);
\foreach \xo in {.25,.75,...,3.25}
\draw[xshift=1.2cm,color=yellow] plot[domain=\xo:4] ({1/4*(\x-\xo)^2},\x);
\foreach \xo in {0,.5,...,3.25}
\draw[xshift=1.2cm,color=blue] plot[domain=\xo:4] ({1/4*(\x-\xo)^2},\x);
\foreach \xo in {-3.25,-2.75,...,-0.25}
\draw[xshift=1.2cm,color=yellow] plot[domain=0:4] ({1/4*(\x-\xo)^2},\x);
\foreach \xo in {-3,-2.5,...,-0.25}
\draw[xshift=1.2cm,color=blue] plot[domain=0:4] ({1/4*(\x-\xo)^2},\x);
\foreach \xo in {.25,.75,...,3.25}
\draw[xshift=1.2cm,color=blue] plot[domain=\xo:4] ({-1/4*(\x-\xo)^2},-\x+4);
\foreach \xo in {0,.5,...,3.25}
\draw[xshift=1.2cm,color=yellow] plot[domain=\xo:4] ({-1/4*(\x-\xo)^2},-\x+4);
\foreach \xo in {-3.25,-2.75,...,-0.25}
\draw[xshift=1.2cm,color=blue] plot[domain=0:4] ({-1/4*(\x-\xo)^2},-\x+4);
\foreach \xo in {-3,-2.5,...,-0.25}
\draw[xshift=1.2cm,color=yellow] plot[domain=0:4] ({-1/4*(\x-\xo)^2},-\x+4);
\end{scope}
\draw[->,color=black] (-1.5,-2.5) -- (5.5,-2.5) node[above] {$t$};
\draw[->,color=black] (2,-2.5) -- (2,-1) node[right] {$\phi(z,t)$};
\draw[xshift=1.2cm,color=violet] plot[domain=0:1.25] ({\x^2},\x-2.5);
\draw[xshift=1.2cm,color=violet] plot[domain=-1.25:0] ({\x^2},-\x-2.5);
\phantom{\draw[xshift=1.2cm,color=blue] plot[domain=-1.25:0] ({\x^2},\x-2.5);}
\draw[color=red,style= very thick] (2,0)--(2,4);
\end{tikzpicture}
\begin{tikzpicture}[x=.6cm, y=.6cm, style= thick]
\draw[->,color=black] (-1.5,0) -- (5.5,0) node[above] {$t$};
\draw[->,color=black] (2,0) -- (2,5) node[right] {$z$};
\begin{scope}
\clip (-1, -1) rectangle (5, 5);
\foreach \xo in {0,0.25,...,3.25}
\draw[xshift=1.2cm,color=orange] plot[domain=\xo:4] ({1/4*(\x-\xo)^2},\x);
\foreach \xo in {-3.25,-3,...,-0.25}
\draw[xshift=1.2cm,color=blue] plot[domain=0:4] ({1/4*(\x-\xo)^2},\x);
\foreach \xo in {0,0.25,...,3.25}
\draw[xshift=1.2cm,color=orange] plot[domain=\xo:4] ({-1/4*(\x-\xo)^2},\x);
\foreach \xo in {-3.25,-3,...,-0.25}
\draw[xshift=1.2cm,color=blue] plot[domain=0:4] ({-1/4*(\x-\xo)^2},\x);
\end{scope}
\draw[->,color=black] (-1.5,-2.5) -- (5.5,-2.5) node[above] {$t$};
\draw[->,color=black] (2,-2.5) -- (2,-1) node[right] {$\phi(z,t)$};
\draw[xshift=1.2cm,color=violet] plot[domain=0:1.25] ({\x^2},\x-2.5);
\draw[xshift=1.2cm,color=violet] plot[domain=-1.25:0] ({\x^2},\x-2.5);
\end{tikzpicture}
\begin{tikzpicture}[x=.6cm, y=.6cm, style= thick]
\draw[->,color=black] (-1.5,0) -- (5.5,0) node[above] {$t$};
\draw[->,color=black] (2,0) -- (2,5) node[right] {$z$};
\begin{scope}
\clip (-1, -1) rectangle (5, 5);
\foreach \xo in {0,0.25,...,3.25}
\draw[xshift=1.2cm,color=orange] plot[domain=\xo:4] ({1/4*(\x-\xo)^2},-\x+4);
\foreach \xo in {-3.25,-3,...,-0.25}
\draw[xshift=1.2cm,color=blue] plot[domain=0:4] ({1/4*(\x-\xo)^2},-\x+4);
\foreach \xo in {0,0.25,...,3.25}
\draw[xshift=1.2cm,color=orange] plot[domain=\xo:4] ({-1/4*(\x-\xo)^2},-\x+4);
\foreach \xo in {-3.25,-3,...,-0.25}
\draw[xshift=1.2cm,color=blue] plot[domain=0:4] ({-1/4*(\x-\xo)^2},-\x+4);
\end{scope}
\draw[->,color=black] (-1.5,-2.5) -- (5.5,-2.5) node[above] {$t$};
\draw[->,color=black] (2,-2.5) -- (2,-1) node[right] {$\phi(z,t)$};
\draw[xshift=1.2cm,color=violet] plot[domain=0:1.25] ({\x^2},-\x-2.5);
\draw[xshift=1.2cm,color=violet] plot[domain=-1.25:0] ({\x^2},-\x-2.5);
\end{tikzpicture}
\caption{Illustration of Examples~\ref{Ex:1},~\ref{Ex:2}. Characteristics are drown for three particular continuous distributional solution to the equation $\phi_{z}+[\phi^{2}/2]_{t}=\sgn(t)/2$. Below the graphs of the corresponding functions $\phi$ are depicted.}
\label{fig:1}
\end{figure}

Before outlining Theorem~\ref{T:othertheorem} we exemplify other features mentioned above by similar examples.

\begin{example}[Figure~\ref{fig:1}]{
\label{Ex:2}
Let $\om=(0,1)\times(-1,1)$, and choose \[\f(z,t)=-\sgn t\sqrt{|t|}.\] 
Again
$$\gf\f=\dede{z}\f+\dede{t}\frac{\f^2}{2}=\left\{\begin{array}{ll}
1/2 & {\rm if}\,t>0 \\
-1/2 & {\rm if}\,t\leq 0
\end{array}\right.=:w(z,t).$$
One easily sees that characteristics do collapse in an essential way.
Considering instead
\[
\f(z,t)=\sgn t\sqrt{|t|}
\] 
characteristics do split in an unavidable way.
Therefore, while it is proved in~\cite{GC} that in Example~\ref{Ex:1} one can choose for changing variable a flux which is better than a generic other---\emph{the regular Lagrangian flow}---we are not always in this case. This is the main reason why we refer in our first change of variables to a monotonicity property.
}
\end{example}

We have now motivated the further study for the stronger statement of Theorem~\ref{T:othertheorem}.
In order to prove it, we consider also the weaker concept of \emph{Lagrangian solution}: the idea is that the reduction on characteristics is not required on \emph{any} characteristic, but on a set of characteristics composing the change of variables $\chi$ that one has chosen.
Exhibiting a suitable set of characteristics for the change of variables $\chi$ is part of the proof.
Roughly, an $\mfrl$-representative $w_{\chi}$ for the source of the ODEs related to $\chi$ is provided by taking the $z$-derivative of $\phi(z,\chi(z,\tau))$, which is $\phi$ evaluated along the characteristics of the Lagrangian parameterisation; by construction, it coincides with the second $y$-derivative of $\chi(z,\tau)$. Here we have hidden the fact that we need to come back from $(z,\tau)$ to $(z,t)$, a change of variable that is not single valued, and not surjective since the second derivative was defined only almost everywhere. We overcome the difficulty choosing any value of the second derivative when present, and showing that it suffices.
However, if one changes the set of characteristics in general one arrives to a different function $w_{\chi'}\in \mfrl(\omega;\R)$. 

We have called \emph{broad solution} a function which satisfies the reduction on every characteristic curve.
In order to have this stronger characterization, we give a different argument borrowed from \cite{abc}.
We define a universal source term $\hat w$ in an abstract way, by a selection theorem, at each point where there exists a characteristic curve with second derivative, without restricting anymore to a fixed family of characteristics providing a change of variables.
After showing that this is well defined, we have provided a universal representative of the intrinsic gradient of $\phi$. In cases as Examples~\ref{Ex:1},~\ref{Ex:2} it extends the one, defined only almost everywhere, provided by Rademacher theorem.

\paragraph{\bf Outline of the paper} The paper is organized as follows. In Section~\ref{S:hnrecalls} we recall basic notions about the Heinsenberg groups. In Section~\ref{S:pderecalls} we fix instead notations relative to the PDE, mainly specifying the different notions of solutions we will consider. One of them will involve a change of variables, for passing to the Lagrangian formulation, which is mainly matter of Section~\ref{S:existencepartialpar} and it is basically concerned with classical theory on ODEs. Then Appendix~\ref{S:extPar} also explains how to extend a partial change of variables of that kind to become surjective, and provides a counterexample to its local Lipschitz regularity; it is improved in~\cite{abc} showing that it is not in general of Bounded variation. In Section~\ref{S:equivalence} we prove the equivalence among the facts that either a continuous function $\phi$ describes a Lipschitz graph or it is a distributional solution to the PDE $\nabla^{\f}\f=w$, $w\in \rmLinf$.
The further equivalencies are finally matter of Section~\ref{S:furthereq}.

With some simplification, we can illustrate the main connections by the following papillon.
As mentioned above, there is also a connection with the existence of smooth approximations (see the proof of Theorem~\ref{T:converseDafermos} by mollification and \cite{pinamonti} with a more geometric procedure).

\begin{displaymath}
\xymatrix{
	{\text{Def.~\ref{D:deflip}}}
      \ar@{.}[d]
  &  {\text{Def.~\ref{D:Lagrsol}}}
     \ar@{.}[dd]
  &{\text{Def.~\ref{D:broadsol}}}
      \ar@{.}[d]
 \\
 *+[F]{\text{Continuous Distributional}}
	\ar@/^0pc/[rr]_{\text{Th.\ref{T:univselection}+L.\ref{L:DafLipschitz}}}
	\ar@/^1pc/[dr]_{\text{Cor.\ref{C:lipordistrlagr} (n=1)}}
	\ar@<1ex>[dd]^{\text{L.\ref{L:distrLip}}}
  &  &*+[F]{\text{Continuous Broad}}
 	\ar@/^1pc/[dl]^{\text{L.\ref{L:globchi}}}
	\ar@/^1pc/@{.>}@<3ex>[dd]_{\text{clear }}
 \\
   &*+[F]{\text{Continuous Lagrangian}}
 	\ar@/^1pc/[ul]^{\text{Cor.\ref{T:converseDafermos}}}
 	\ar@/^1pc/[ur]_{\text{Th.\ref{T:univselection}}}
  	\ar@{.>}[dr]^{\text{clear}}
\\
  *+[F]{\text{Intrinsic Lipschitz}}
 	\ar@{->}[ur]|{\text{Cor.\ref{C:lipordistrlagr} (n=1)}}
 	\ar@/^1pc/[uu]^{\text{L.\ref{L:lipDistr}}}
	& &
      *+[F]{\text{Continuous Broad*}}
       	\ar@/^1pc/@{.>}[ul]^{\text{L.\ref{L:globchi}}}
\\
  {\text{Def.~\ref{D:distribsol}}}
      \ar@{.}[u]
	& &
      {\text{Def.~\ref{D:broad*sol}}}
      \ar@{.}[u]
        }
\end{displaymath}

{\bf A remark about the identification of the source terms.}
The intrinsic gradient is unique.
Suppose $\phi$ is both intrinsic Lipschitz and a broad solution: then by definition at points of intrinsic differentiability the intrinsic gradient must coincide with $\hat w$.
As a broad solution is also a Lagrangian solution, at those points it must coincide also with $\bar w$.
By Rademacher theorem, these functions are therefore identified Lebesgue almost everywhere, and by Lemma~\ref{L:lipDistr} they are representative of the $\rmLinf(\omega;\R^{2n-1})$-function which is the source term of the balance law.
\\
We conclude remarking again that this identification Lebesgue almost everywhere is not enough for fixing them, as both Lagrangian solutions and broad solutions require the right representative of the source also on some of the points of non-intrinsic differentiability, not captured by the source term in the distributional formulation (Fig.~\ref{fig:1}).
By Rademacher theorem, instead, we have that the source term in the broad formulation not only fixed dirctly the source term in the Lagrangian formulation, but also in the distributional one.
Actually, it suffices the source of a Lagrangian formulation in order to fix directly the one in the distributional one.



{\it Acknowledgements.} We warmly thank Stefano Bianchini and Giovanni  Alberti for useful discussions and important suggestions, in particular on the subject of Section~\ref{S:furthereq}.

\section{Sub-Riemannian Heisenberg Group}
\label{S:hnrecalls}

{\bf Definition: a noncommutative Lie group.}
We denote the points of $\Hn\equiv\C^n\times\R\equiv\R^{2n+1}$
by
\[
  P=(x,t),\qquad x\R^{2n},\ t\in\R.
\]
If $P=(x,t)$, $Q=(x',t')\in \Hn$ and $r>0$, the group
operation reads as
\begin{equation}
P\cdot Q:=\left (x+x',t+t'+\frac{1}{2}\sum_{j=1}^n\Big( x_jx'_{j+n}-x_{j+n}x'_j \Big)\right).
\end{equation}
The group identity is the origin 0 and one has
$[x,t]^{-1}=[-x,-t]$. In $\mathbb H^n$ there is a natural one
parameter group of non isotropic dilations
$\delta_{r}(P):=[rx,r^2t]$ , $r>0$.

The group $\h^n$ can be endowed with the homogeneous norm
\begin{equation*}
\| P\|_\infty:=\max\{|x|, |t|^{1/2}\}
\end{equation*}
and with the left-invariant and homogeneous distance
\begin{equation*}
   d_{\infty}(P,Q):=\| P^{-1}\cdot Q\|_\infty.
\end{equation*}
The metric $d_\infty$ is equivalent to the standard
Carnot-Carath\'eodory distance. It follows that the Hausdorff
dimension of $(\Hn,d_{\infty})$ is $2n+2$, whereas
its topological dimension is $2n+1$.

The Lie algebra $\mathfrak{h}_n$ of left invariant vector fields
is (linearly) generated by
\begin{equation}\label{defcampiHn}
  X_j =\dede{x_j}-\frac{1}{2}x_{j+n}\dede{t},
  \quad X_{j+n}=\dede{x_{j+n}}+\frac{1}{2}x_j\dede{t},
  \quad j=1,\dots,n,
  \quad T=\dede{t}
\end{equation}
and the only nonvanishing commutators are
\begin{equation*}
[X_j,X_{j+n}] = T, \qquad j=1,\dots n.
\end{equation*}

{\bf Horizontal fields and differential calculus.}
We shall identify vector fields and associated first order differential operators; thus the vector fields  $X_1,\dots,X_{2n}$ generate a vector bundle on $\Hn$, the so called {\it horizontal} vector bundle $\HH$ according to the notation of Gromov (see \cite{gromov2}), that is a vector subbundle of $\hbox{\rm T}\Hn$, the tangent vector bundle of $\Hn$. Since each fiber of $\HH$ can be canonically identified with a vector subspace of $\R^{2n+1}$, each section $\varphi$ of $\HH$ can be identified with a  map $\varphi:\Hn\rightarrow\R^{2n+1}$. At each point $P\in \Hn$ the horizontal fiber is indicated as $\HH_{P}$ and each fiber can be endowed with the scalar  product $\langle{\cdot},{\cdot}\rangle_{P}$ and the associated norm $\|\cdot\|_P$ that make the vector fields $X_1,\dots,X_{2n}$ orthonormal.

\begin{definition}
A real valued function $f$, defined on an open set
$\Omega\subset\Hn$, is said to be of class
$\mathrm C^1_\h(\Omega)$ if $f\in \mathrm C^0(\Omega)$ and the distribution
\[
  \nabla_\h f: =(X_1f,\dots,X_{2n}f)
\]
is represented by a continuous function. 
\end{definition}

\begin{definition}
We shall say that $S\subset \Hn$ is an {\sl $\h$-regular hypersurface} if for every $P\in S$ there exist an open ball
$U_{\infty}(P,r)$ and a function $f\in\CH (U_{\infty}(P,r))$ such that
\begin{itemize}
\item[i:] $S\cap U_{\infty}(P,r) = \{Q\in U_{\infty}(P,r): f(Q)=0\}$;
\item[ii:] $\nabla_{\h}f(P)\neq 0$.
\end{itemize} 
\end{definition}
The {\em horizontal normal} to $S$ at $P$ is $\nu_S(P):=-\dfrac{\nabla_\h f (P)}{|\nabla_\h f (P)|}. $

\begin{notation}\label{notation}{\rm
Recalling that we denote a point of $\Hn$ as $P=(x,t)\in\Hn$, where $x=(x_1,\dots,x_{2n})\in\R^{2n}$, $t\in\R$, it is convenient to introduce the following notations:
\begin{itemize}
\item We will denote a point of $\R^{2n}\equiv\R^{2n-1}\times\R$ as $A=(z,t)\in\Hn$, where $z\in\R^{2n}$, $t\in\R$. If $z=(z_1,\dots,z_{2n-1})\in\R^{2n-1}$, for given $i=1,\dots,2n-1$, we denote $\hat z_i=(z_1,\dots,z_{i-1},z_{i+1},\dots,z_{2n})\in\R^{2n-2}$. When we will use this notation, we also denote a point $A=(z,t)\in\R^{2n}$ as $A=(z_i,\hat z_i,t)\in\R\times\R^{2n-2}\times\R$.
\item Let $B\subset\R^{2n}$, for given $\hat z_i\in\R^{2n-2}$ and $t\in\R$ we will denote $B_{\hat z_i,t}:=\{s\in\R:\ (s,\hat z_i,t)\in B\};$
\item Let $B\subset\R^{2n}$, for given $z\in \R^{2n-1}$ we will denote$B_{z}:=\{t\in\R:\ (z,t)\in B\}.$
\item Let $B\subset\R^{2n}$, for given $\hat z_i\in\R^{2n-2}$ we will denote $B_{\hat z_i}:=\{(s,t)\in\R^2:\ (s,\hat z_i,t)\in B\};$
\item $\mathbb W=\{(x,t)\in\Hn:\ x_1=0\}$. A point $A=(0,z_1,\dots,z_{2n-1},t)\in\W$ will be identified with the point $(z,t)\in\R^{2n}$.
\item Let $\om\subset\W$, for given $z_n\in\R$ we will denote $\om_{z_n}:=\{(\hat z_n,t)\in\R^{2n-1}:\ (z_n,\hat z_n,t)\in\om\}$.
\end{itemize}
}\end{notation}


\begin{definition}
A set $S\subset\Hn$ is an $X_1$-graph if there is a function $\phi:\omega\subset\W\rightarrow\mathbb V$ such that 
$S=G_{\mathbb H,\phi}^1(\omega):=\{A\cdot\phi(A)e_1:A\in\omega\}.$
\end{definition}

Let us recall the following results proved in \cite{frasersercas}.

\newcommand{\primacomp}{y_{1}}
\newcommand{\tcomp}{t}
\begin{theorem}[Implicit Function Theorem]
\label{DiniTheorem}
Let $\Omega$ be an open set in $\Hn$, $0\in\Omega$, and let $f\in\CH 
(\Omega)$ be such that $X_1f>0$. 
Let
$S:=\{(x,t)\in\Omega: f(x,t)= 0\};$ 
then there exist a connected open neighborhood $\mathcal{U}$ of $0$ and a unique continuous 
function $\phi:\omega\subset\W\rightarrow [-h,h]$ such that $S\cap 
\overline{\mathcal{U}}  = \F(\omega)$, where $h>0$ and $\F$ is the map defined 
as 
\[
\omega\ni(z,\tcomp)
\mapsto
\F(z,\tcomp)=(z,\tcomp)\cdot\phi(z,\tcomp)e_1
\]
and given explicitly by
\begin{equation*}
\begin{array}{ll}
\F(z,\tcomp)=\left(\f(z,\tcomp),z_1,\dots,z_{2n-1},\tcomp-\displaystyle\frac{z_n}{2}\f(z,\tcomp)\right) & \mbox{if }n\geq 2\\
\F(z,\t)=\left(\f(z, \tcomp),z,\tcomp-\displaystyle\frac{z}{2}\f(z, \tcomp )\right) & \mbox{if }n=1.
\end{array}
\end{equation*}
\end{theorem}

Let $n\geq2$, $A_0=(z_1^0,\dots,z_{2n-1}^0,t^0)\in\mathbb R^{2n}$ and define
\begin{equation*}
\begin{array}{cl}
I_r(A_0)&:=\displaystyle\left\{(z,t)\in\mathbb R^{2n}:\,|z_n-z_n^0|<r,\, \sum_{i=1}^{2n-1} [(z_i-z_i^0)^2]< r^2,\,|t-t^0|<r\right \}\,.
\end{array}
\end{equation*}

When $n=1$ and $A_0=(z^0,t^0)\in\mathbb R^2$ let
$$I_r(A_0):=\displaystyle{\left\{(z,t)\in\mathbb R^2:\,|z-z^0|<r,\,|t-t^0|<r\right\}}\,.$$


Following \cite{ascv,tesibigolin,vitttesi} we define the graph quasidistance $\df$ on $\om$. We set $\mathbb O_1:=\{(x,t)\in\Hn:\ x_1=0,t=0\}$, $\mathbb T:=\{(x,t)\in\Hn:\ x_1=0,\dots, x_{2n}=0\}$

\begin{definition}\label{fidistanzadef2708}
For $A=(z_n,\hat z_n,\tcomp),\,A'=(z_n',\hat z_n',\tcomp')\in\omega$ we define
\begin{equation}\label{defindf}
\df(A,A'):=\|\pi_{\mathbb O_1}(\F(A)^{-1}\cdot\F(A'))\|_\infty+
\|\pi_{\mathbb T}(\F(A)^{-1}\cdot\F(A'))\|_\infty
\end{equation}
Following the notations \ref{notation} we have explicitly if $n\geq 2$
$$\df(A,A')=|(z_n',\hat z_n')-(z_n,\hat z_n)|+\left|\tcomp'-\tcomp-\frac{1}{2}(\f(A)+\f(A'))(z_n'-z_n)+\sigma(\hat z_n,\hat z_n')\right|^{1/2};$$
where $\sigma(\hat z_n,\hat z_n')=\frac{1}{2}\sum_{j=1}^{n-1}(z_{j+n}z'_j-z_jz'_{j+n})$.
If $n=1$ and $A=(z_2,\tcomp),A'=(z_1',\tcomp')\in\omega$ we have $$\df(A,A')=|z_1'-z_1|+\left|\tcomp'-\tcomp-\frac{1}{2}(\f(A)+\f(B))(z_1'-z_1)\right|^{1/2}.$$
\end{definition}

An intrinsic differentiable structure can be induced on $\W$ by means of $d_\f$, see \cite{ascv,tesibigolin,vitttesi}.
We remind that a map $L:\W\to\R$ is $\W$-linear if it is a group homeomorphism and $L(rz,r^{2}t)=rL(z,t)$ for all $r>0$ and $(z,t)\in\W$. 
We remind then the notion of $\gf$-differentiablility.

\begin{definition}\label{defiWfdiff}
Let $\f:\omega\subset\W\rightarrow\re$ be a fixed continuous function, and let
$A_0\in\omega$ and $\p:\omega\rightarrow\re$ be given. 
\begin{itemize}
\item 
We say that $\p$ is $\gf$-differentiable at $A_0$ if there is a unique $\W$-linear functional $L:\W\rightarrow\re$ such that
\begin{equation}\label{definWfdiffernziabilita}
\lim_{A\rightarrow A_0}\frac{\p(A) -\p(A_0) - L(A_0^{-1}\cdot A)}{\df(A_0,A)} = 0.
\end{equation}

\item
We say that $\p$ is uniformly $\gf$-differentiable at $A_0$ if there is a unique $\h$-linear functional $L:\W\rightarrow\re$ such that
\begin{equation}\label{defdiL_fperW_fdiffunif}
\lim_{r\rightarrow 0}\sup_{A,B\in I_r(A_0)\atop{A\neq B}}\left\{\frac{|\p(B) -\p(A) - L(B^{-1}\cdot A)|}{\df(A,B)}\right\}=0
\end{equation}
\end{itemize}
\end{definition}
We will denote $L=d_\W\f(A_0)$.
If $\f$ is uniformly $\gf$-differentiable at $A_0$, then $\f$ is $\gf$-differentiable at $A_0$.

In \cite{ascv} it has been  proved that each $\h$- regular graph $\F(\omega)$ admits an  intrinsic gradient $\nabla^\f\f\in \mathrm C^0(\omega;\R^{2n})$, in  sense of distributions, which shares a lot of properties with the Euclidean gradient. Since $\mathbb W=\exp(\text{span}\{X_2,\dots,X_{2n},T\})$, where the vector fields $X_2,\dots,X_{2n},T$ must be understood as restricted to $\W$, it is possible to define the differential operators given, in  sense of distributions, by
\begin{equation}
 \label{defnablaphi2}
\begin{array}{l}
   \displaystyle\Wf\f:=X_{n+1}\f+\frac{1}{2}T(\f^2)
=\dede{x_{n+1}}\phi+\frac{1}{2}\dede{t}(\f^2) ,\\
   \displaystyle\nabla^\f\f:=\left\{\begin{array}{ll}
(X_2\f,\dots,X_n\f,\Wf\f,X_{2+n}\f,\dots,X_{2n}\f) & \text{if
}n\geq 2\\ \Wf\f & \text{if }n=1
\end{array}\right.\,.\end{array}
\end{equation}
According with the notation introduced in \ref{notation},
we also  denote by $\nabla^{\f}:=( \nabla_1^{\f},\dots,\nabla_{2n-1}^{\f})$ the family of vector fields on $\R^{2n}$ defined by $$\nabla^{\f}_{j}\f:=\left\{\begin{array}{ll}
X_{j+1}\f=\dfrac{\partial\f}{\partial z_j}-\dfrac{z_{j+n}}{2}\dfrac{\partial\f}{\partial t} & 
j=1,\dots,n-1\\
\Wf\f=X_{n+1}\f+\,\f\, T\f= \dfrac{\partial\f}{\partial z_n}-\f\dfrac{\partial\f}{\partial t} &  j=n\\
X_{j+1}\f=\dfrac{\partial\f}{\partial z_j}+\dfrac{z_{j+n}}{2}\dfrac{\partial\f}{\partial t} &  j=n+1,\dots,2n-1
\end{array}\right.$$

The following characterizations were proved in \cite{ascv,tesibigolin,BSC2}.
The definitions of broad* and distributional solution of the system $\gf\f=w$ are recalled in Section~\ref{S:pderecalls}.

\begin{theorem}
\label{miotheorem}
Let $\omega\subset \W\equiv\R^{2n}$ be an open set and let
$\f:\omega\rightarrow\R$ be a continuous function. Then
\begin{enumerate}
\item The set $S:=\F(\omega)$  is an $\h$-regular surface and
$\nu_S^{1}(P)<0$ for all $P\in S$, where
$\nu_S(P)=(\nu_S^{1}(P),\dots,\nu_S^{2n}(P))$ is the horizontal
normal to $S$ at $P$.
\end{enumerate}
is equivalent to each one of the following conditions:
\begin{enumerate}
\item[(ii)] There exists $w\in \mathrm C^0(\omega;\R^{2n-1})$ and a family
$(\f_\epsilon)_{\epsilon>0}\subset \mathrm C^1(\omega)$ such that, as
$\epsilon\to0^+$,
\begin{equation*}
 \f_\epsilon\rightarrow\f\quad\mbox{and}
 \quad\nabla^{\f_\epsilon}\f_\epsilon\rightarrow w\quad
 \text{in }\,\rmLinf_{\loc}(\omega)\,,
\end{equation*}
and 
$\nabla^{\phi}\phi=w$ in $\omega\,,$
 in sense of distributions.
 
\item[(iii)] There exists $w\in \mathrm C^0(\omega;\R^{2n-1})$ such that $\phi$ is a  {\it broad* solution} of the system $\gf\f=w$.

\item[(iv)] There exists $w\in \mathrm C^0(\omega;\R^{2n-1})$ such that $\f$ is a distributional solution of $\gf\f=w$.

\item[(v)] $\f$ is uniformly $\gf$-differentiable at $A$ for all $A\in \om$.
\end{enumerate}
\end{theorem}

\begin{remark}{\rm It follows that, for given $\f:\om\to\R$ such that $S=\F(\om)$ is $\h$-regular, then at any point $A\in\om$ the differetial $d_\W\f(A):\W\equiv\R^{2n}\to\R$ can be represented in terms of the intrinsic gradient $\gf\f(A)$. More precisely (see \cite{ascv,pinamonti})
$$d_\W\f(A)(z)=\left\{\begin{array}{ll} 
\gf\f(A)z & n=1 \\ \displaystyle\sum_{i=1}^{2n-1}\gf_i\f(A)z_i & n\geq 2 \end{array}\right.$$
}\end{remark}

\begin{definition}\label{D:gfderivative}
Let $\f:\om\to\R$ be a continuous function and let $A\in\om$ be given. We say that $\f$ admits $\gf_i$-derivative if there exists $\alpha\in\R$ such that for each integral curve $\Upsilon:(-\de,\de)\to\om$ of $\gf_i$ with $\Upsilon(0)=A$ $$\exists\ \dede{s}\f(0)=\lim_{s\to0}\dfrac{\f(\Upsilon(s))-\f(\Upsilon(0))}{s}=\alpha,$$
and $\alpha=\gf_i\f(A)$.
\end{definition}

\begin{remark}
Notice that if $S=\F(\om)$ is an $\h$-regular hypersurface then $\f$ admits $\gf_i$-derivative at $A$
for every $A\in\om$ for $i=1,\dots,2n-1$ (see \cite{ascv}, \cite{vitttesi}). 
\end{remark}

{\bf Introduction to the concern of this paper.}

Let us now introduce the concept of intrinsic Lipschitz function and intrinsic Lipschitz graph.

\begin{definition}
\label{D:deflip}
Let $\f:\om\subset\W\to\R$. We say that $\f$ is an intrinsic Lipschitz continuous function in $\om$ and write $\f\in \mathrm{ Lip}_\W(\om)$, if there is a constant $L>0$ such that
\begin{equation}
\label{deflip}
|\f(A)-\f(B)|\leq L\df(A,B)\qquad \forall A,B\in\om
\end{equation}
Moreover we say that $\f$ is a locally intrinsic Lipschitz function in $\om$ and we write 
$\f\in \mathrm{ Lip}_{\W,\loc}(\om)$ if $\f\in \mathrm{ Lip}_\W (\om)$ for every $\om'\Subset\om$.
\end{definition}

We remark that when $\phi$ is intrinsic Lipschitz, then there exists $C>0$ such that
\[
\frac{1}{C}\df(A,B) 
\leq
d_{\infty}(\Phi(A),\Phi(B))
\leq
C\df(A,B)
\qquad
\forall A,B\in\omega.
\]
In particular, the graph distance $d_{\phi}$ is also equivalent to the Carnot-Carath\'eodory distance restricted to the corresponding points on the graph of the Lipschitz intrinsic hypersurface. This means that $\phi$ is Lipschitz continuous also in the classical sense when evaluated on any fixed integral curve of the vector field $\Wf$, while it is $1/2$-H\"older on the lines where $t$ is fixed.

%
In \cite{pinamonti} is proved the following  characterization for intrinsic Lipschitz functions. 

\begin{theorem}\label{T:tesipinamonti}
Let $\om\subset\W$ be open and bounded, let $\f:\om\to\R$. Then the following are equivalent:
\begin{enumerate}
\item[(i)]  $\f\in \mathrm{ Lip}_{\W,\loc}(\om)$
\item[(ii)] there exist $\{\f_k\}_{k\in \N}\subset \mathrm C^{\infty}(\om)$ and $w\in (\rmLinf_{\loc}(\om))^{2n-1}$ such that $\forall \om'\Subset\om$ there exists $C=C(\om')>0$ such that
\begin{enumerate}
\item[(ii1)] $\{\f_k\}_{k\in \N}$ uniformly converges to $\f$ on the compact sets of $\om$;
\item[(ii2)] $|\nabla^{\f_k}\f_k(A)|\leq C \quad \mathcal L^{2n}$-a.e. $x\in \om',$  $k\in\N$;
\item[(ii3)] $\nabla^{\f_k}\f_k(A)\to w(A)\quad \mathcal L^{2n}$-a.e. $A\in \om$.
\end{enumerate}
\end{enumerate}
Moreover if (ii) holds, then $\gf\f(A)=w(A)$ $\mathcal L^{2n}$-a.e. $A\in \om$.
\end{theorem}

Let us finally recall the following Rademacher type  Theorem, proved in \cite{frasersercaslip}.
\begin{theorem}\label{T:rademacher}
If $\f\in \mathrm{ Lip}_{\W}(\om)$ then $\f$ is $\gf$-differentiable for $\mathcal L^{2n}$-a.e $A\in\om$.
\end{theorem}




\section{Different Solutions of the Intrinsic Gradient Differential Equation}
\label{S:pderecalls}

Even when $w\in \mathrm C^{0}(\omega)$, where $\omega$ is an open subset of $ \W\equiv\R^{2}$, the equation
\begin{equation}
\label{E:balance}
\phi_{z}(z,t)+ \left[\frac{\phi^{2}(z,t)}{2}\right]_{t}=w(z,t)\qquad{\rm in\ }\om
\end{equation}
allows in general for discontinuous solution. However, it is the case $n=1$ of the system~\eqref{defnablaphi2}
\begin{equation}\label{E:balancen}
\gf\f=w
\qquad 
\Leftrightarrow
\qquad
\begin{cases}
X_{j+1}\phi = w_{j}
&j=1,\dots,n-1,n+1,\dots,2n-1
\\
W^{\phi}\phi=w_n
&
\\
\end{cases}
\end{equation}
where $w\in\mathrm C^0(\om,\R^{2n-1})$ and this system, by Theorem~\ref{miotheorem}, describes an $\h$-regular surface $S:=\F(\omega)$ which is an $X_{1}$-graph.
Since we want to study in the present paper intrinsic Lipschitz graphs, then we do not require anymore the continuity of $w$ but we allow  $w\in\rmLinf(\omega;\R^{2n-1})$. Notwithstanding that, the continuity of $\phi$ remains natural.

There are a priori different notions of \emph{continuous} solutions $\phi:\omega\to\R$. We recall some of them in this section: distributional, Lagrangian, broad, broad*. 
All of them will finally coincide.

After giving in the present sections the definitions for all $n$, we will focus in the next one the analysis on the non-linear equation in the case $n=1$, which conveys the attention on the planar case~\eqref{E:balance}. We will remind this reduction by adopting often the variables $(s,\tau)$ instead of $(y,t)$. The generalization to other cases $n\geq2$ of most of the lemmas is straightforward, because the fields $X_j$ and $Y_j$ are linear. It is not basically in Lemma~\ref{L:gfunztx}, where we prefer taking advantage of the continuity of $\chi$; however, we have no reason to prove it in full generality.
%
%

We recall that in general solutions are not smooth, even if we assume the continuity---see e.g.~Example~\ref{E:fillholesNoGlobLip} below.
The equation is then interpreted in a distributional way.

\begin{definition}[Distributional solution]\label{D:distribsol}
A continuous function $\phi:\omega\to\R$ is a \emph{distributional solution} to~\eqref{E:balancen} if for each $\varphi\in \mathrm C^{\infty}_c(\omega)$
\begin{equation}
\int_{\omega}\f\, \gf_{j}\varphi\,d\mathcal L^{2n}=-\int_{\omega}w_j\varphi\,d\mathcal L^{2n},\qquad j=1,\dots,n-1,n+1,\dots 2n-1
\end{equation}
and \begin{equation*}
\int_{\omega}\left(\f\dede{z_n}\varphi+\frac{1}{2}\f^2\dede{t}\varphi\right)\,d\mathcal L^{2n}=\,-\int_{\omega}w_{n}\varphi\,d\mathcal L^{2n}.
\end{equation*}
\end{definition}


We consider now different versions for the Lagrangian formulation of the PDE. The first one somehow englobes a choice of trajectories for passing from Lagrangian to Eulerian variables, and imposes the evolution equation on these trajectories.

\begin{definition}[Lagrangian parameterisation]
\label{D:Lagrparam}
A family of \emph{partial Lagrangian parameterisations} associated to a continuous function $\phi:\omega\to\R$ and to the system~\eqref{E:balancen} is a family of couples $(\tilde\omega_i,\chi_i)$ ($i=2,\dots,2n$) with $\tilde\omega_i\subset\R^{2n}$ open sets and $\chi_i=\chi_i(\xi,t)=\chi_i(\xi_i,\hat\xi_i,t):\tilde \omega_i\to \R$ Borel functions such that  for each $i=1,\dots,2n-1$
\begin{itemize}
\item[(L.1):] the map $\Upsilon_i:\tilde\om_i\to\R^{2n}$,
$\Upsilon_i(\xi,t)= (\xi,\chi_i(\xi,t))$ is valued in $\omega$;
\item[(L.2):] for every $\xi\in\R^{2n-1}$, the function
$\tilde\om_{i,\xi}\ni t\mapsto\chi_i(\xi,t)$ is nondecreasing;
\item[(L.3):] for every $\hat\xi_i\in\R^{2n-2},\ t\in\R$, for every $(s_1,s_2)\subset\tilde\om_{i,\hat\xi_i,t}$, the function $(s_1,s_2)\ni s\mapsto\Upsilon_i(s,\hat\xi_i,t)$ is absolutely continuous and 
\begin{equation}
\label{E:generalODE}
\dede{s}\chi_i(s,\hat\xi_i,t)=\gf_i(\Upsilon(s,\hat\xi_i,t))\quad{\rm a.e.\ }s\in(s_1,s_2).
\end{equation}
\end{itemize}
We call it a family of \emph{(full) Lagrangian parameterisations} if $\chi_i:\tilde\om_\xi\to\om_\xi$ is onto the section $\omega_\xi$ for all $\xi$.
\end{definition}

We remark again that we emphasized in this definition the nonlinear PDE of the system: a Lagrangian parameterisaiton provides a covering of $\omega$ by characteristic lines for that equation. Indeed, a covering by characteristic lines of the other equations is immediately given by an expression like
\[
\chi_i(z_1,\dots,z_{2n-1},t)
= 
\begin{cases}
t-\tfrac{z_{i+n}}{2} z_i & i=1,\dots,n-1\\
t+\tfrac{z_{i-n}}{2} z_i & i=n+1,\dots,2n-1 
\end{cases}\]
Moreover, the reduction along characteristics for the linear equations, and thus the equivalence between Lagrangian and distributional solution, holds with less technicality.

\begin{definition}
A family of (partial) parameterisations $(\tilde\omega_i, \chi_i)$ extends the family of (partial) parameterisations $(\tilde {\omega}_i', \tilde {\chi}_i')$, we denote $(\tilde {\omega}'_i, \tilde {\chi}'_i)\preceq(\tilde\omega_i,\chi_i)$, if there exists a family of Borel injective maps
\[
J_i:\tilde\omega'_i\ni(z,\tau) \mapsto (z, j(z,\tau))\in\tilde\omega_i
\qquad\text{such that}\qquad
\chi_i\circ J_i= \tilde\chi'_i\quad\forall\ i=1,\dots,2n-1.
\]
When $(\tilde\omega'_i, \tilde\chi'_i)\preceq(\tilde\omega_i,\chi_i)$ and $(\tilde\omega_i,\chi_i)\preceq(\tilde\omega'_i, \tilde\chi'_i)$ they are called equivalent. 
\end{definition}

\begin{remark}
The notion of Lagrangian parameterisation given above does not consist in a different formulation for the notion of regular Lagrangian flow in the sense by Ambrosio-Di Perna-Lions (see~\cite{Crippa} for an effective presentation). Particles are really allowed both to split and to join, therefore in particular the compressibility condition here is not required, while instead we have a monotonicity property.
\end{remark}

\begin{notation}{\rm
When we need to distinguish letters, we denote with $\bar\cdot$ functions defined on $\omega$ but possibly related to a parameterisation, with $\tilde\cdot$ functions defined on $\tilde \omega$, and with $\hat\cdot$ functions defined on $\omega$ not related to specific parameterisations.}
\end{notation}

Notice that a full Lagrangian parameterisation is continuous in the two variables for free: indeed, e.g.~for $n=1$, by monotonicity one has that for each $s$
\[
\sup_{\tau' < \tau} \chi(s,\tau')\leq \chi(s,\tau)\leq \inf_{\tau' >\tau}\chi(s,\tau').
\]
By the surjectivity then equality must hold. Considering the Lipschitz continuity in the other variables one gains the joint continuity in $(s,\tau)$.
The same holds for $n>1$ provided that $\chi$ is continuous on the hyperplane $z_n=0$, since the argument above gives the continuity only on the planes where $\hat z_n$ is constant. We do not mind about continuity in $\hat z_n$.

Before giving the notion of Lagrangian solution, we recall that a set $A\subset\R^{n}$ is \emph{universally measurable} if it is measurable w.r.t.~every Borel measure. Universally measurable sets constitute a $\sigma$-algebra, which includes analytic sets. A function $f:\R^n\to\R$ is said universally measurable if it is measurable w.r.t.~this $\sigma$-algebra. In particular, it will be measurable w.r.t.~any Borel measure.
\\
Notice that Borel counterimages of universally measurable sets are universally measurable. Then the composition $\varphi\circ\psi$ of any universally measurable function $\varphi$ with a Borel function $\psi$ is universally measurable. This composition would be nasty if $\varphi $ were just Lebesgue measurable.
\\
Since restrictions of Borel functions on Borel sets are Borel, all the terms in the following definition are thus meaningful.

\begin{definition}[Lagrangian solution]
\label{D:Lagrsol}
A continuous function $\phi:\om\to\R$ is a \emph{Lagrangian solution} of~\eqref{E:balancen} if there exists a family of Lagrangian parameterisations $(\tilde\omega_i, \chi_i)$ ($i=1,\dots,2n-1$), associated to $\phi$ and \eqref{E:balancen}, and a family of universally measurable functions $\bar w_{\chi_i}:\omega\to\R$ ($i=1,\dots,2n-1$), such that $\forall\hat\xi_i\in\R^{2n-2}$, $\forall t\in\R$, $\forall(s_1,s_2)\subset\tilde\om_{\xi_i,t}$, it holds that
\begin{itemize}
\item[(LS1):] the function $(s_1,s_2)\ni s\mapsto\phi(\Upsilon(s,\xi_i,t))$ is absolutely continuous and
\begin{equation}
\label{E:evolutionphi}
\dede{s}\f(\Upsilon_i(s,\hat\xi_i,t))=\bar w_{\chi_i}(\Upsilon_i(s,\hat\xi_i,t))
\end{equation}
where $\Upsilon$ is given in {\rm (L.1)} of definition \ref{D:Lagrparam}.
\item[(LS2):] If $\f$ admits $\gf_i$-derivative at $A\in\om$, then $\bar w_{\chi_i}(A)=\gf_i\f(A)$ $\forall\ i=1,\dots,2n-1$.
\end{itemize}
\end{definition}

We are going to prove in Section~\ref{Ss:lagrdistr} that $\phi$ above is a distributional solution to~\eqref{E:balancen} provided that the function $\bar w_{\chi} \in \mfrl(\omega;\R^{2n-1})$ is a pointwise representative of the source term $w\in\rmLinf(\omega, \Ll^{2n})$.
The fact that we do not require in the definition that $\bar w_{\chi} $ and $w$ represent the same function is just a notational convenience for the proofs below.

We give now the strongest notion of solution: the evolution equation is imposed on every trajectories.

\begin{definition}[Broad solution]
\label{D:broadsol}
A continuous function $\phi:\omega\to\R$ is a \emph{broad solution} of~\eqref{E:balancen} if there exists a universally measurable function $\hat w=(\hat w_1,\dots,\hat w_{2n-1})\in \mfrl(\omega;\R^{2n-1})$ such that
\begin{itemize}
	\item[(B.1):] $\hat w=w$ $\mathcal L^{2n}$-a.e. in $\om$;
	\item[(B.2):] $\forall\ i=1,\dots,2n-1$, for every integral curve $\Gamma_i\in \mathrm C^{1}((-\de,\de); \omega)$ of the vector field $\nabla^{\phi}_{i} $ it holds that $(-\delta,\delta)\ni s\mapsto\f(\Gamma_i(s))$ is absolutely continuous and
$$\dede{s}\f(\Gamma_i(s))=\hat w(\Gamma_i(s))\qquad {\rm a.e.}\ s\in (-\de,\de);$$
\item[(B.3):] if $\f$ admits $\gf_i$-derivative at $A\in\om$, then $\hat w_i(A)=\gf_i\f(A)$ $\forall i=1,\dots,2n-1$.
\end{itemize}
\end{definition}

We also remind the intermediate notion of broad* solution introduced in~\cite{BSC1}.

\begin{definition}[Broad* solution]
\label{D:broad*sol}
A continuous function $\phi$ is a \emph{broad* solution} of~\eqref{E:balancen} if for every $A_0\in\omega$ there exists a map, that we will call exponential map, 
\begin{equation}\label{exponential}
\exp_{A_0}(\nabla^\phi_i)(\cdot):[-\de,\de]\times{I_{\de}(A_0)}\rightarrow{I_{\delta_1}(A_0)} \Subset\omega\,,
\end{equation}
where $0<\de<\delta_1$ such that, if $\Gamma_i^A(s)=\exp_{A_0}(s\nabla^{\phi}_i)(A)$ for $i=1,\dots,2n-1$, then
\begin{itemize}
	\item[(E.1):] $\displaystyle{\Gamma_i^A\in\,\mathrm C^1([-\de,\de];\,\R^{2n})}\,,$
	\item[(E.2):] $\displaystyle{\left\{ \begin{array}{ll}
	\dot\Gamma_i^A= {\nabla^{\f}_{i}\circ\Gamma_i^A} \\
	\Gamma_i^A(0)=A
	\end{array}
	\right.\,,}$
	\item[(E.3):] there exists a universally measurable function $\bar w\in \mfrl(\omega;\R^{2n-1})$ such that $\bar w=w$ $\mathcal L^{2n}$-a.e. in $\om$, $\forall\ i=1,\dots,2n-1$ the map $(-\delta,\delta)\ni s\mapsto\f\left(\Gamma^A_i(s)\right)$ is absolutely continuous and
$$\dede{s}\f\left(\Gamma^A_i(s)\right)=\bar w\left(\Gamma^A_i(s)\right)\qquad {\rm a.e.}\ s\in (-\de,\de)\qquad \forall A\in I_\de(A_0);$$
\item[(E.4):] if $\f$ admits $\gf_i$-derivative at $A\in\om$, then $\bar w_i(A)=\gf_i\f_i(A)$ $\forall i=1,\dots,2n-1$.
\end{itemize}
\end{definition}

Being a distributional solution to the PDE, a Lagrangian solution will be in particular a broad* solution with $w:=\bar w$.
Viceversa, if the exponential map related to a broad solution satisfies the relative monotonicity property, then one can prove that the broad* solution is also a Lagrangian solution, with the same $\bar w$. One can moreover derive a procedure for constructing a Lagrangian parameterisation in Lemma~\ref{L:globchi}, where the curves $\gamma(s;\bar s,\bar t)$ should be replaced by the ones of the exponential map.

\section{Existence of Lagrangian parameterisations}
\label{S:existencepartialpar}
As mentioned, we can focus in this section on the case $n=1$. We will give information about the generalization to the case $n\geq2$ in Remark~\ref{R:remarkdimnsez4}.

Given a continuous function $\phi$ on $\clos(\omega)\subset\W$, we defined what we mean by `Lagrangian parameterisation'. Here we show that the definition is non-empty.
It is part of the very classical theory on ODEs with continuous coefficients covering $\omega$ with integral curves of the continuous vector field $(1,\phi)$, see for example the first chapters of~\cite{hartman}. 
In the definition we of Lagrangian parameterisation it is however essential that we cover the region with curves which are ordered and not with generic ones:
we therefore briefly remind below one way of doing it, with also the aim being self-contained.

Firstly, in Lemma~\ref{L:solODE} we show, for familiarising with notations, that if one takes the integral curves through an $s$-section of $\omega$ which are minimal, in the sense that any other curve through that point lies on its right side, then we get a partial Lagrangian parameterisation, as well as it happens for the maximal ones.
Extending it to a full one will be matter of the next section.
A  full Lagrangian parameterisation basically amounts to an order preserving parameterisation, with a \emph{real} valued parameter, of non-crossing curves through each point of the plane.
We immediately give an example of a full one (Lemma~\ref{L:globchi} below). 

Notice that the definition of Lagrangian parameterisation concerns only classical theory on ODEs with continuous coefficients, as therefore Lemmas~\ref{L:solODE},~\ref{L:globchi},~\ref{L:gfunztx},~\ref{L:holderode},~\ref{L:fillinODE}.

\begin{lemma}
\label{L:solODE}
Let $\phi:\mathrm{clos}(\omega)\to\R$ be a continuous function, $(0,0)\in\omega$.
Then there are domains $\tilde \omega_{m}$, $\tilde \omega_{M}$ associated to the functions
\begin{align*}
\chi_{m}(s,\tau)&:=\min\Big\{\gamma(s):\ (r,\gamma(r))\in\clos(\omega),\ \dot\gamma(r)=\phi(r,\gamma(r)),\ \gamma(0)=\tau \Big\}
\\
\chi^{}_{M}(s,\tau)&:=\max \Big\{\gamma(s):\ (r,\gamma(r))\in \clos(\omega),\ \dot\gamma(r)=\phi(r,\gamma(r)),\ \gamma(0)=\tau \Big\}
\end{align*}
for which $(\tilde \omega_{m}, \chi_{m} )$, $(\tilde \omega_{M}, \chi_{M} )$ are partial Lagrangian parameterisaiton relative to $\phi$.
\end{lemma}
\begin{proof}
For every $(\bar s,\bar \tau)\in\omega$ one could consider the minimal and maximal curves satisfying on $\clos(\omega)$ the ODE for characteristics \eqref{E:generalODE} and passing through that point: indeed the functions
\begin{subequations}
\label{E:gammamM}
\begin{align}
\label{E:gammamM1}
\gamma_{(\bar s,\bar \tau)}(s)&:= \min\Big\{\gamma(s):\ (r,\gamma(r))\in\clos(\omega),\ \dot\gamma(r)=\phi(r,\gamma(r)),\ \gamma(\bar s)= \bar \tau\Big\}
\\
\gamma^{(\bar s,\bar \tau)}(s)&:= \max\Big\{\gamma(s):\ (r,\gamma(r))\in\clos(\omega),\ \dot\gamma(r)=\phi(r,\gamma(r)),\ \gamma(\bar s)= \bar \tau\Big\}
\end{align}
\end{subequations}
are well defined, Lipschitz, and because of the continuity of $\phi$ they are still integral curves (\cite{hartman}).

Denoting by $\ri$ the relative interior of a set, we define the domain
\[
	\tilde\omega_{m}=\ri\big\{(s,\tau)\in \R\times\omega_{0}:\ \tau \in \omega_{0},\ (s,\gamma_{(0,\tau)}(s))\in\omega\big\}.
\]
where recall that $\omega_o=\{t\in\R: (0,t)\in\omega\}$.
Since we are assuming that $\omega$ contains the origin, $\tilde\omega_{m} $ is nonempty.
The domain $\tilde\omega_{M}$ is analogous.

Being $\gamma_{(\bar s,\bar \tau)}(s)$, $\gamma^{(\bar s,\bar \tau)}(s)$ Lipschitz solutions to the ODE with continuous coefficients, the functions $\chi_{m}(s,\tau),\chi_{M}(s,\tau)$ in the statement are $\mathrm C^{1}(\omega\cap\{t=\tau\})$ in the $s$ variable for every $\tau$ fixed. 
$\chi_{m}, \chi^{M}$ are jointly Borel in $(s, \tau)$ by continuity in $s$ and monotonicity in $\tau $, monotonicity that now we show.

Notice the semigroup property: for $t,h>0$, for example for~\eqref{E:gammamM1},
\[
\gamma_{(0,\tau)}(s)=:\overline \tau,
\quad
\gamma_{(s, \overline \tau)}(h)=: \bar {\overline \tau},
\quad
\Longrightarrow
\quad
\gamma_{(0, \tau)}(s+h)=\bar {\overline \tau}.
\]
This yields the monotonicity of $\chi_{m}(s,\tau)=\gamma_{(0,\tau)}(s)$ for each $s$ fixed. Indeed, if $\tau_{1}< \tau_{2}$ and $\gamma_{(0, \tau_{1})}(\hat s)\geq \gamma_{(0, \tau_{2})}(\hat s)$ at a certain $\hat s>0$, by the continuity of the curves there exists $0<\bar s\leq \hat s$ when $\gamma_{(0, \tau_{1})}(\bar s)= \gamma_{(0, \tau_{2})}(\bar s)$.
But then the curve
\[
\gamma(s)=
\begin{cases}
\gamma_{(0, \tau_{1})}(s)
& \text{for $s\leq \bar s$}
\\
\gamma_{(0, \tau_{2})}(s)
& \text{for $s\geq \bar s$}
\end{cases}
\]
is a good competitor for the definition of $\gamma_{(0, \tau_{1})}$, which implies
\[
\gamma_{(0, \tau_{2})}(\hat s)= \gamma(\hat s)\geq  \gamma_{(0, \tau_{1})} (\hat s),
\]
and therefore equality.
For $\gamma^{(0, \tau)}$ the argument is similar.
\end{proof}

In Lemma~\ref{L:fillinODE} we show how to make a partial parameterisation $\chi$ surjective: thus we cover $\omega$ by a family of characteristic curves which includes the ones of $\chi$.
In the following lemma, w.l.o.g.~in a simpler setting, we provide instead a full Lagrangian parameterisation, defined at once instead of extending a given one.

\begin{lemma}
\label{L:globchi}
There exists a full Lagrangian parameterisation associated to a continuous function $\phi:(0,1)^{2}\to \R$ which is also continuous on the closure $[0,1]^{2}$.
\end{lemma}
\begin{proof}%
{\sf Step $1$.} 
Let $\phi$ be a continuous function on $[0,1]^{2}$, we want to give a Lagrangian parameterisation for his restriction to $(0,1)^{2}$ as we defined it on an open set. One can assume $\phi$ is compactly supported in $(s,\t)\in [0,1]\times (0,1)$.
If not, one can extend it to a compactly supported function $\bar \phi$ on $[0,1]\times (-1,2)$: restricting the Lagrangian parameterisation $(\tilde \omega^{\bar\phi},\chi^{\bar\phi})$ for $\bar \phi$, defined as described below, to the open set
\[
\tilde \omega := \{(s,\tau)\in \tilde \omega^{\bar\phi}:\ \chi^{\bar\phi}(s,\tau)\in (0,1) \},
\] 
one will get a Lagrangian parameterisation for $\phi$.
The assumption of $\t$ compactly supported in $(0,1)$ implies that there are two characteristics, one starting from $(0,0)$ and one from $(0,1)$, which satisfy $\dot \gamma(s) = \phi(s,\gamma(s))\equiv 0$. This means respectively $\gamma(s)\equiv 1$ and $\gamma(s)\equiv 0$ for each $s\in[0,1]$.

{\sf Step $2$.} 
After this simplification, we associate to each point $(\bar s, \bar t)\in[0,1]^{2}$ a curve $\gamma(s;\bar s, \bar t)$ which is minimal forward in $s$, maximal backward:
\[
\gamma(s;\bar s, \bar t)=
\begin{cases}
\gamma_{(\bar s, \bar t)}(s)
&\text{$s\geq \bar s$,}
\\
\gamma^{(\bar s, \bar t)}(s)
&\text{$s<\bar s$.}
\end{cases}
\]
where $\gamma_{(\bar s, \bar t)}(s)$, $\gamma^{(\bar s, \bar t)}(s)$ were defined in~\eqref{E:gammamM}.
Notice that the curves $\gamma(\cdot;\bar s,\bar t)$ are defined on the whole $[0,1]$. Let us denote by
\[ 
\mathcal C:=\{\gamma(\cdot;\bar s,\bar t):\ [0,1]\to[0,1],\ (\bar s,\bar t)\in[0,1]^2\}.
\]
We will endow $\mathcal C$ by the topology of uniform convergence on $[0,1]$ and the following total order relation
\begin{equation}\label{E:lemma4.2eq1}
\big\{\gamma(\cdot; s_{1}, t_{1})\} \preceq \big\{\gamma(\cdot; s_{2}, t_{2}) \big\}
\qquad\Longleftrightarrow\qquad
\gamma(s; s_{1}, t_{1}) \leq \gamma(s; s_{2}, t_{2})\quad \forall s\in[0,1].
\end{equation}
Let us denote by $\mathcal C^*$ the closure of $\mathcal C\subset C^0([0,1])$ endowed with the topology of uniform convergence.
We prove in the rest of this step the following claims:
\begin{equation}\label{E:lemma4.2eq2}
\mathcal C^*\text{ is compact};
\end{equation}
\begin{equation}\label{E:lemma4.2eq3}
\text{the total order relation \eqref{E:lemma4.2eq3} still applies in }\mathcal C^*;
\end{equation}
\begin{equation}\label{E:lemma4.2eq4}
\mathcal C^*\text{ is connected};
\end{equation}
\begin{equation}\label{E:lemma4.2eq5}
\mathcal C^* \text{ is still a family of characteristic curves for }\phi,
\end{equation}
i.e. for each $\gamma\in\mathcal C^*$ $\dot\gamma(s)=\phi(s,\gamma(s))$ $\forall\ s\in[0,1]$.

{Proof of \eqref{E:lemma4.2eq2}}.
Since $\mathcal C$ is a family of equi-Lipschitz continuous and bounded functions of $C^0([0,1])$, then, from Arzel\`a-Ascoli's theorem $\mathcal C^*$ is compact. 

Proof of \eqref{E:lemma4.2eq3}.
Let $\gamma,\tilde\gamma\in\mathcal C^*$ and let us prove that $\gamma\preceq\tilde\gamma$ or $\tilde\gamma\preceq\gamma$. By definition there are two sequences $\{\gamma_h\}_h,\{\tilde\gamma\}_h\subset\mathcal C$ such that $ \gamma_h\to\gamma$ and $\tilde\gamma_h\to\tilde\gamma$ uniformly on $[0,1]$. Assume $\gamma\neq\tilde\gamma$, then there is $s_0\in[0,1]$ such that 
\[
\tilde\gamma(s_0)<\gamma(s_0)\quad{\rm or}\quad\gamma(s_0)<\tilde\gamma(s_0).
\]
For instance, let $\tilde\gamma(s_0)<\gamma(s_0)$ and let us prove that
\begin{equation}\label{E:lemma4.2eq6}
\tilde\gamma(s)<\gamma(s)\quad\forall\ s\in[0,1].
\end{equation}
Let $0<\epsilon<\dfrac{\gamma(s_0)-\tilde\gamma(s_0)}{2}$, there exists $\bar h=\bar h(\epsilon)$ such that
\begin{equation}\label{E:lemma4.2eq7}
|\gamma(s)-\gamma_h(s)|<\epsilon\ {\rm and}\ |\tilde\gamma(s)-\tilde\gamma_h(s)|<\epsilon\quad\forall\ s\in[0,1],\ \forall h<\bar h.
\end{equation}
From \eqref{E:lemma4.2eq7} it follows that
\[
\tilde\gamma_h(s_0)<\tilde\gamma(s_0)+\epsilon<\gamma(s_0)-\epsilon<\gamma_h(s_0),\quad\forall h>\bar h.
\]
Because $\tilde\gamma_h$ and $\gamma_h$ are ordered, it follows that $\tilde\gamma_h(s)\leq\gamma_h(s)$ $\forall\ s\in[0,1]$, $\forall\ h>\bar h$. Passing to the limit as $h\to\infty$ in the previous inequality, we get \eqref{E:lemma4.2eq6}.

Proof of \eqref{E:lemma4.2eq4}. By contradiction, suppose that $C^*=C_1\cup C_2$, with $C_1$ and $C_2$ non empty, closed sets in $C^0([0,1])$.
It is well-know that, from \eqref{E:lemma4.2eq2} and \eqref{E:lemma4.2eq3}, for each subset $\mathcal A\subset\mathcal C^*$ there exists the least upper bound (or supremum) and greatest lower bound (or infimum) of $\mathcal A$; we will denote respectively by $\sup\mathcal A$, $\inf\mathcal A$. Thus let
\[
l_1:=\inf\mathcal C_1\ \leq\ L_1:=\sup\mathcal C_1,\qquad
l_2:=\inf\mathcal C_2\ \leq\ L_2:=\sup\mathcal C_2.
\]
Because $\mathcal C_1$ and $\mathcal C_2$ are closed, $l_1,L_1\in\mathcal C_1$ and $l_2,L_2\in\mathcal C_2$. Since $\mathcal C_1\cap\mathcal C_2=\emptyset$, we have $L_1\preceq l_2$ or $L_2\preceq l_1$. Assume, for instance, that $L_1\preceq l_2$. Then $L_1(s)\leq l_2(s)$ $\forall\ s\in[0,1]$ and $L_1(s_0)<l_2(s_0)$ for a suitable $s_0\in[0,1]$. Let 
\[
\bar s:=s_0,\qquad \bar t:=\dfrac{L_1(s_0)+l_2(s_0)}{2},\quad \gamma(s):=\gamma(s;\bar s,\bar t)\quad s\in[0,1].
\]
By definition, $\gamma\in\mathcal C\subseteq\mathcal C^*$, but $L_1\preceq\gamma\preceq l_2$ and therefore we have a contradiction since $\gamma\notin\mathcal C_1\cup \mathcal C_2=\mathcal C^*$.

Proof of \eqref{E:lemma4.2eq5}. Let $\gamma\in\mathcal C^*$, then by definition there exists a sequence $\{\gamma_h\}_h\subset\mathcal C$ such that $\gamma_h\to\gamma$ uniformly in $[0,1]$. Because
\[
\gamma_h(s)-\gamma_h(0)=\int_0^s\phi(\sigma,\gamma_{h}(\sigma))\ d\sigma\qquad \forall\sigma\in[0,1],\ \forall\ h,
\]
passing to the limit as $h\to\infty$ in the previous identity we get \eqref{E:lemma4.2eq5}.

{\sf Step $3$.} 
Let us now consider the map $\theta:\mathcal C^*\to\R$ defined by
\[  \theta(\gamma):=\sum_{k=0}^{\infty} \dfrac{1}{2^k}\gamma(r_k)  \]
where $\{r_k\}_{k\in\N}$ is an enumeration of $\mathbb Q\cap[0,1]$. Notice that $\theta$ satisfies the following properties:
\begin{equation}\label{E:lemma4.2eq8}
\theta\text{ is continuous};
\end{equation} 
\begin{equation}\label{E:lemma4.2eq9}
\theta\text{ is strictly order preserving, that is } \theta(\gamma_1)<\theta(\gamma_2)\ {\rm if} \  \gamma_1<\gamma_2;
\end{equation} 
\begin{equation}\label{E:lemma4.2eq10}
\theta(\mathcal C^*)=[0,2];
\end{equation} 
\begin{equation}\label{E:lemma4.2eq11}
\text{there exists }\theta^{-1}:[0,2]\to\mathcal C^*\ \text{continuous}.
\end{equation} 

Indeed \eqref{E:lemma4.2eq8} and \eqref{E:lemma4.2eq9} immediately follow by the definition of $\theta$. Equality \eqref{E:lemma4.2eq10} follows by \eqref{E:lemma4.2eq4} and \eqref{E:lemma4.2eq8}; \eqref{E:lemma4.2eq11} is a consequence of \eqref{E:lemma4.2eq2}, \eqref{E:lemma4.2eq9} and \eqref{E:lemma4.2eq10}. 
Eventually let us consider the map $\chi:[0,1]\times[0,2]\to[0,1]$ 
\[
\chi(s,\tau): = \theta^{-1}(\tau)(s) \ \text{if } (s,\tau)\in[0,1]\times[0,2].
\]
Then, from \eqref{E:lemma4.2eq11}, $\chi$ is continuous. Moreover the map $\Upsilon:(0,1)\times(0,2)\to(0,1)^2$, defined by $\Upsilon(s,\tau):=(s,\chi(s,\tau))$, turns out to be a full Lagrangian parameterisation associated to $\phi:(0,1)^2\to\R$.
\end{proof}

\begin{remark}\label{R:remarkdimnsez4}{\rm
Notice that the same construction works with more variables, considering analogous characteristic curves $\gamma(s;\bar s,\bar{\hat{z}}_n,\bar t)$---having the same order relation at $\bar{\hat{z}}_n$ frozen---and the relative parameterisation given by $\chi(s,\hat z_n,\tau):=\gamma(s; \bar s, \hat z_n, \bar t)$ with $\tau=\theta(\gamma(s; \bar s, \hat z_n, \bar t))$ defined as above.
The continuity in $\hat z_n$ however is not guaranteed.}
\end{remark}

\section{Equivalence between Distributional solutions and Intrinsic Lipschitz Condition.}
\label{S:equivalence}
In this section we prove Theorem \ref{T:firsttheorem}, without dimensional restrictions.

\subsection{Some properties of distributional solutions}

Preliminary we highlight here two properties of continuous distributional solutions $\phi(z,t)$ to the problem $\gf\f=w $. 
In particular we need regularity results of the solution along the characteristics lines $\gamma$ of the fields $\gf_j$. 
In the case of the non linear field $\Wf$, the integral curves of $\dot\gamma(s)=\phi(s,\hat z_n,\gamma(s))$ exist by the continuity and boundedness of $\phi$.

\begin{lemma}\label{L:DafLipschitz} Let $\om\subset\W$ be an open set.
A continuous distributional solution $\f:\om\to\R$ to $\gf\f=w$ is $\|w\|_{\infty}$-Lipschitz along any characteristic line $\gamma:[-\delta,\delta]\to \R$ satisfying
$$\dot \gamma(s)=\f(s,\hat z_n,\gamma(s))\qquad s\in[-\delta,\delta],\ \hat z_n\ {\rm fixed}.$$
\end{lemma}

\begin{proof}
In the same way in Dafermos \cite{Dafermos} we obtain for $a,b\in (-\de,\de)$ and $\hat z_n$ fixed
\begin{equation}\label{daf0623}
\int_{\gamma(b)}^{\gamma(b)+\epsilon}\phi(b,\hat z_n,t)dt-\int_{\gamma(a)}^{\gamma(a)+\epsilon}\phi(a,\hat z_n,t)dt-\int_{a}^{b}\int_{\gamma(s)}^{\gamma(s)+\epsilon}w(s,\hat z_n,t)\ dt\,ds=
\end{equation}
$$=-\int_{a}^{b}[\phi(s,\hat z_n,\gamma(s)+\epsilon)-\phi(s,\hat z_n,\gamma(s))]^2\,ds\leq0$$
and then
$$\int_{\gamma(b)}^{\gamma(b)+\epsilon}\phi(b,\hat z_n,t)dt-\int_{\gamma(a)}^{\gamma(a)+\epsilon}\phi(a,\hat z_n,t)dt\leq \int_{a}^{b}\int_{\gamma(s)}^{\gamma(s)+\epsilon}w(s,\hat z_n,t)\ dt\,ds\leq \|w_{n}\|_{\rmLinf(\om)}(b-a)\epsilon$$
Dividing by $\epsilon$ and getting to the limit to $\epsilon\to0$ we obtain
$$\phi(b,\hat z_n,\gamma(b))-\phi(a,\hat z_n,\gamma(a))\leq\|w_{n}\|_{\rmLinf(\om)}(b-a).$$
The opposite inequality is obtained in a similar way integrating on the left of the characteristic.
\end{proof}

We obtain the same result of Lemma \ref{L:DafLipschitz} for the linear fileds $X_j,Y_j$ following the same proof.

\begin{lemma}
\label{L:DafLipschitzbis} 
Let $\om\subset\W$ be an open set.
A continuous distributional solution $\f:\om\to\R$ to $\gf\f=w$ is $\|w_{j}\|_{\rmLinf(\om)}$-Lipschitz along any characteristic line $\Gamma:[-\delta,\delta]\to \om$ satisfying
$$\dot \Gamma(s)=\gf_j\circ \Gamma(s)\qquad j=1,\dots,n-1,n+1,\dots, 2n-1.$$
\end{lemma}

We pass now to the H\"older continuity in the other variable.

\begin{lemma}
\label{L:holderode} 
Let $f\in \mathrm C^0(\om)$ be such that of all $\tau\in\R$ there are $\gamma:[-\delta,\delta]\to \R$ satisfying
$$\begin{cases}\dot \gamma(s)=f(s,\gamma(s)) &s\in[-\delta,\delta]\\ \gamma(0)=\tau\end{cases}$$
and that
$$|f(s,\gamma(s))-f(s',\gamma(s'))|\leq L|s-s'|.$$
Then we have
$$|f(0,\tau_1)-f(0,\tau_2)|\leq 2 \sqrt{2L }\sqrt{|\tau_1-\tau_2|}
\qquad \text{for $|\tau_1-\tau_2|\leq r_0 $,\quad $0<r_0<\frac{\de^4}{16}$}
.$$
\end{lemma}

\begin{proof} Let us denote 
$$
\beta:= L,\quad \alpha:= {\max\{2\sqrt{2L},r_0^{1/4}\}}, \quad f_0(\tau)=f(0,\tau).
$$ 
Let us observe that
$\dfrac{\beta}{\alpha^2}\leq \, \dfrac{1}{8}.$

By contradiction, let us assume there exist  ${-\de \le\,\t_2<\,\t_1\le\,\de}$, $0<\,\bar r<\,r_0<\delta$ such that
\begin{equation}\label{hoelderode4}
0<\,|\t_1-\t_2|\le\,\bar r
\end{equation}
\begin{equation}\label{hoelderode5}
\frac{|f_0(\t_1)- f_0(\t_2)|}{\sqrt{\t_1-\t_2}}>\,\alpha.
\end{equation}

The idea of the proof is the following: the Lipschitz condition in the hypothesis is an upper bound on the second derivative of the mentioned curves $\gamma$. Therefore, if their first derivative wants to change it takes some time in $s$. If we assume that at $s=0$ the first derivative differs at two points $\tau_{1},\tau_{2}$ at least of the ratio~\eqref{hoelderode5}, then the the relative curves $\gamma_1, \gamma_2 $ starting from those points must meet soon. However, at the time they meet the first derivative did not have the time to change enough to become equal, providing an absurd.

Let us introduce curves $\gamma_1, \gamma_2$ such that for $i=1,2$, $s\in {(-\de,\de)}$ 
$$\dot \gamma_i(s)=f(s,\gamma_i(s)),$$
$$
\gamma_{i}(0):=\t_i.
$$
We observe that, by our Lipschitz assumption, $\dfrac{d}{ds}f(s,\gamma_i(s)):=h_i(s)$ is a function bounded by $L$.
Therefore we can represent each $\gamma_i(s)$ for each $s\in [-\delta,\delta]$
as 
\begin{equation}\label{hoelderode2}
 \gamma_i(s)=\,\t_i+\,\displaystyle\int_{0}^s f(s,\gamma_{i}(\sigma))\,d\sigma
\end{equation}
$$=\,\t_i+\,g(\t_i)\,s+\,\int_{0}^s \int_0^\sigma h_i(z)\,dz\ d\sigma\quad\forall s\in [-\de,\de]\,.$$

In particular by the second equality in \eqref{hoelderode2}, for $0\leq s\leq\de$,
\begin{equation}\label{hoelderode9}
\gamma_1(s)-\gamma_2(s)\leq \t_1-\t_2+(f_0(\t_1)-f_0(\t_2))s+2\beta s^2
\end{equation}
for $s\in[-\de,\de]$.
By \eqref{hoelderode5} we get
\begin{equation}\label{hoelderode6}
f_0(\t_1)- f_0(\t_2)<\,-\alpha\,\sqrt{\t_1-\t_2}
\end{equation}
or
\begin{equation}\label{hoelderode7}
f_0(\t_1)- f_0(\t_2)>\,\alpha\,\sqrt{\t_1-\t_2}
\end{equation}
Let us prove now that if \eqref{hoelderode6} holds then there exists $0<\,s^*<\,\de$ such that
\begin{equation}\label{hoelderode8}
\gamma_1(s^*)=\gamma_2(s^*).
\end{equation}

Let $\bar s:=\displaystyle 2\,\frac{\sqrt{\t_1-\t_2}}{\alpha}$ then

\begin{equation}\label{hoelderode10}
\bar s\in [0,\de], \quad {\gamma_1(s^*)\leq\gamma_2(s^*)}\,.
\end{equation}
Indeed by \eqref{hoelderode4} and the definition of $\alpha$, $\bar s=\, 2\,\frac{\sqrt{\t_1-\t_2}}{\alpha}\le\, 2\, (\t_1-\t_2)^{1/4}\le\, 2\, \bar r^{1/4}\le\, \de$. On the other hand by \eqref{hoelderode9} (with $s=\bar s$),  \eqref{hoelderode6} gives 
\begin{equation*}
\begin{array}{cl}
\displaystyle \gamma_1(\bar s)-\gamma_2(\bar s)&\le\,\t_1-\t_2-\, 2\,(\t_1-\t_2)+\dfrac{8\beta}{\alpha^2}\,(\t_1-\t_2)\\
&=\,(\t_1-\t_2)\, \left(-1+\dfrac{8\beta}{\alpha^2}\right){\leq }0
\end{array}
\end{equation*}
Then \eqref{hoelderode10} follows. Let 

\begin{equation*}
s^*:=\sup\{s\in[0,\de_2]:\gamma_1(s)>\gamma_2(s)\}
\end{equation*} 

then by \eqref{hoelderode8}  $0<s^*<\bar s\le\, \de$ and it satisfies \eqref{hoelderode8}.

If \eqref{hoelderode7} holds we can repeat the argument reversing the $s$-axis getting that there exist $-\de<\,s*<\;0$ such that \eqref{hoelderode8} still holds. 
 
Let us prove now that
\begin{equation}\label{hoelderode11}
f(s*,\gamma_1(s*))\neq f(s*,\gamma_2(s*)),
\end{equation}
then a contradiction and the thesis will follow. Indeed, for instance, let us assume \eqref{hoelderode6}. Then by \eqref{hoelderode2} and the bound on $h_i$
$$
f(s*,\gamma_{1}(s*))-f(s*,\gamma_{2}(s*))=f_0(\t_1)-f_0(\t_2)+\int_0^{s*}h_1(\sigma)-h_2(\sigma)d\sigma\leq
$$
$$\leq f_0(\t_1)-f_0(\t_2)+2\beta\,s*\leq\,f_0(\t_1)-f_0(\t_2)+2\beta\,\bar s
$$
$$\leq\,-\alpha\,\sqrt{\t_1-\t_2}+2\, \frac{2\beta}{\alpha}\,\sqrt{\t_1-\t_2}\leq
$$
$$
\leq 2\,\alpha\,\sqrt{\t_1-\t_2}\left[-\frac{1}{2}+\frac{2\beta}{\alpha^2}\right]\,.
$$ 
Therefore we get that
$$f(s*,\gamma_{1}(s*))-f(s*,\gamma_{2}(s*)))<0$$ and \eqref{hoelderode11} follows.
\end{proof}

\subsection{Proof of the equivalence}
\label{Ss:firsttheorem}

We are now able to give the proof of Theorem \ref{T:firsttheorem}. We distinguish the two different implications in the following two lemmas.

\begin{lemma}
\label{L:distrLip}
Let $\om\subset\W$ be an open set.
If $\f$ is a distributional solution of $\gf\f=w$, then $\f\in \mathrm{ Lip}_{\W,\loc}(\om)$.
\end{lemma}
\begin{proof}
Let $B,B'\in I_{\de_0}(A)$ for a sufficiently small $\de_0$. For $n\geq 2$ let $\overline X, \Wf$ be the vector fields given by
\begin{equation*}
\overline X:=\sum_{j=1\atop {j\neq n+1}}^{2n-1}(z_j'-z_j)\widetilde X_j,\qquad\Wf:=\dede{z_n}+\f\dede{t}.
\end{equation*}
Define
\begin{eqnarray*}
B^\ast &:=& \exp(\overline X)(B)\\
&=& B\star(0,(z_1'-z_1,\dots,z_{n-1}'-z_{n-1},z_{n+1}'-z_{n+1},\dots,z_{2n-1}'-z_{2n-1}),0)\\
&=& (z_n,\hat z_n',t-\sigma(\hat z',\hat z))\\
B'' &:=& \exp((z_n'-z_n)\Wf)(B^\ast)=(z_n',\hat z_n',t'')\quad\mbox{(for a certain $t''$);}
\end{eqnarray*}
observe that $B^\ast$ and $B''$ are well defined. For $n=1$, $\overline X$ is not defined and we set $B^\ast=B$ and $B'':=\exp((z_n'-z_n)\Wf)(B)=(z_n',t'')$. 

We have to show that there exists $L>0$ such that 
\begin{equation}\label{tesilipheisenberg}
|\f(B)-\f(B')|\leq L\df(B,B').
\end{equation}
We have
\begin{equation}\label{daci1agenerale1}
|\f(B)-\f(B')| \leq |\f(B)-\f(B^\ast)|+|\f(B^\ast)-\f(B'')|+|\f(B'')-\f(B')|
\end{equation}

Notice that from  Lemma \ref{L:DafLipschitzbis}
\begin{equation}\label{app19lugeq2tris}
|\f(B)-\f(B^\ast)|\leq \|(w_1,\dots,w_{n-1},w_{n+1},\dots,w_{2n-1})\|_{\rmLinf(\om,\R^{2n-2})}\ |\hat z_n'-\hat z_n|
\end{equation}
and then $|\f(B)-\f(B^\ast)|\leq \df(B,B')$, 
from  Lemma \ref{L:DafLipschitz}
\begin{equation}\label{app19lugeq2bis}
|\f(B^\ast)-\f(B'')|\leq \|w_{n}\|_{\rmLinf(\om)}\ |z_n'-z_n|
\end{equation}
and then $|\f(B^\ast)-\f(B'')|\leq \df(B,B')$.

By \eqref{app19lugeq2bis} we can apply Lemma \ref{L:holderode} and obtain
\begin{equation}\label{app19lugeq2}
|\f(B')-\f(B'')|\leq C \sqrt{|t'-t''|}
\end{equation}

Let's observe that
\begin{equation}\label{equaz3.9n}
\begin{array}{cl}
&\displaystyle|t'-t''|\\
=& \displaystyle\Bigl|t'-t+\sigma(\hat z_n',\hat z_n)-\int_0^{z_n'-z_n}\f(\exp(s\Wf)(B^\ast))\,ds\Bigr|\vspace{0.1cm}\\
\leq& \displaystyle\Bigl|t'-t-\frac{1}{2}(\f(B')+\f(B))(z_n'-z_n)+\sigma(\hat z_n',\hat z_n)\Bigr|+\\
& \displaystyle+\frac{1}{2}\Bigl|(\f(B')+\f(B))(z_n'-z_n)-2\int_0^{z_n'-z_n}\f(\exp(s\Wf)(B^\ast))\,ds \Bigr|\vspace{0.2cm}\\
\leq& \displaystyle\df(B',B)^2+\frac{1}{2}|\f(B')-\f(B'')|\ |z_n'-z_n|+\frac{1}{2}|\f(B^\ast)-\f(B)|\ |z_n'-z_n|+\\
& \displaystyle+\frac{1}{2}\Bigl|[\f(B'')+\f(B^\ast)](z_n'-z_n)-2\int_0^{z_n'-z_n}\f(\exp(s\Wf)(B^\ast)\,ds \Bigr|\vspace{0.2cm}\\ 
=:& \displaystyle\df(B',B)^2+R_1(B',B)+R_2(B',B)+R_3(B',B).
\end{array}
\end{equation}
For the case $n=1$ we arrive to \eqref{equaz3.9n} with the same line (it is sufficient to follow the same steps ``erasing'' the term $\sigma(\hat z_n',\hat z_n)$).

Now we want to prove that for all $\epsilon>0$ there is a $\de_\epsilon\in]0,\de_0]$ such that, for $\de\in]0,\de_\epsilon[$,
\begin{equation}\label{equaz3.11n}
R_1(B',B)\leq |z_n'-z_n|^2+\epsilon|t'-t''|
\end{equation}
for all $B',B\in I_\de(A)$ 
and that there exist 
$C_1,C_2>0$ such that
\begin{eqnarray}
& & R_2(B',B) \leq C_2 \df (B',B)^2\label{equaz3.12n}\\
& & R_3(B',B)\leq C_1|z_n'-z_n|^2\label{equaz3.10n}
\end{eqnarray}
for all $B',B\in I_{\de_0}(A)$,

 These estimates are sufficient to conclude: in fact, choosing $\epsilon:=1/2$ and using \eqref{equaz3.9n}, \eqref{equaz3.10n}, \eqref{equaz3.11n} and \eqref{equaz3.12n}, we get 
\begin{equation*}
|t'-t''| \leq \df(B',B)^2+C_1|z_n'-z_n|^2+|z_n'-z_n|^2+|t'-t''|/2+C_2\df(B',B)^2
\end{equation*}
whence
\begin{equation*}
|t'-t''|^{1/2}\leq C_3 \df(B,B')
\end{equation*}
which is $|\f(B')-\f(B'')|\leq \df(B,B')$ and then the thesis \eqref{tesilipheisenberg}.

By \eqref{app19lugeq2} we obtain
$$R_1(B',B)\leq 2C \sqrt{|t'-t''|}\ |z_n'-z_n|\leq \epsilon |t'-t''|+\frac{1}{\epsilon} |z_n'-z_n|^2,$$
 whence \eqref{equaz3.11n} follows. 
 
Observe that \eqref{equaz3.12n} follows from $R_2(B,B')=0$ if $n=1$, and from
\begin{eqnarray*}
R_2(B',B) &=& |z_n'-z_n||\f(B)-\f(B^\ast)|\\
&\leq& 2C_2|z_n'-z_n||\hat z_n'-\hat z_n|\leq C_2|(z_n'-z_n,\hat z_n'-\hat z_n)|^2\leq C_2\df(B',B)^2
\end{eqnarray*}
if $n\geq 2$.
Finally, for $s\in[-\de_0,\de_0]$ we can define 
\begin{equation}\label{defindig}
g(s):=2\int_0^s \f(\exp(r\Wf)(B^\ast))\,dr-\bigl[ \f\bigl(\exp(s\Wf)(B^\ast)\bigr)+\f(B^\ast)\bigr]s;
\end{equation}
We have $$g(s)=2\int_0^s \big[\f(\exp(r\Wf)(B^\ast))-\f(B^{\ast})\big]\,dr-\bigl[ \f\bigl(\exp(s\Wf)(B^\ast)\bigr)-\f(B^\ast)\bigr]s=o(s^2)$$
because $(-\de_0,\de_0)\ni s\mapsto\f(\exp(s \Wf)(B^\ast))$ is Lipschitz. Therefore \eqref{equaz3.10n} follows with $s=z_n'-z_n$. 
\end{proof}

\begin{corollary}\label{holder} Let $\omega\subset\W$ be an open set.
If $\f\in \mathrm C^0(\om)$ is a distributional solution of $\gf\f=w$ in $\omega$, then
\begin{equation}\label{tesicamillo}
|\f(A)-\f(B)|\leq C\sqrt{|A-B|}\qquad \forall\ A,B\in\omega
\end{equation}
\end{corollary}

\begin{lemma}
\label{L:lipDistr}
If  $\f\in \mathrm{ Lip}_{\W,\loc}(\om)$, then $\f$ is a distributional solution of $\gf\f=w$ in $\om$.
\end{lemma}
\begin{proof}
 By Theorem~\ref{T:tesipinamonti} there exist $\{\f_k\}_{k\in \N}\subset \mathrm C^{\infty}(\om)$,  such that $\f_k$ uniformly converges to $\f$ on the compact sets of $\om$, $|\nabla^{\f_k}\f_k(A)|\leq C$ $\mathcal L^n$-a.e. $A\in\om$ for every $k\in\N$ and $\nabla^{\f_k}\f_k(A)\to w(A)$ $\mathcal L^n$-a.e. $A\in\om$.
Therefore, denoting $w_k:=\nabla^{\f_k}\f_k$, we have
for every $k\in\N$ and for every $\varphi\in \mathrm C^{\infty}_c(\om)$
$$\int_{\omega}\f_kX_j\varphi\,d\mathcal L^{2n}=-\int_{\omega}w_{j,k}\varphi\,d\mathcal L^{2n}, \qquad j=1,\dots,n-1,n+1,\dots,2n-1$$
$$\int_{\omega}\left(\f_k\dede{z_n}\varphi+\frac{1}{2}\f_k^2\dede{t}\varphi\right)\,d\mathcal L^{2n}=\,-\int_{\omega}w_{n,k}\varphi\,d\mathcal L^{2n}.$$
Getting to the limit for $k\to\infty$ we obtain
$$\int_{\omega}\f X_j\varphi\,d\mathcal L^{2n}=-\int_{\omega}w_{j}\varphi\,d\mathcal L^{2n}, \qquad j=1,\dots,n-1,n+1,\dots,2n-1$$
$$\int_{\omega}\left(\f\dede{z_n}\varphi+\frac{1}{2}\f^2\dede{t}\varphi\right)\,d\mathcal L^{2n}=\,-\int_{\omega}w_n\varphi\,d\mathcal L^{2n}$$
i.e. $\f$ is a distributional solution of the problem $\gf\f=w$.
\end{proof}

\section{Further Equivalences} 
\label{S:furthereq}

In the previous section we established the equivalence between
\[
\text{`$\phi:\omega\subset\W\to\R$ is intrinsic Lipschitz continuous'}
\]
and
\[
\text{`$\phi\in \mathrm C^{0}(\omega;\R)$ and there exists $w\in\mathrm  L^{\infty}(\omega;\R^{2n-1})$ such that $\nabla^{\phi}\phi=w$ in $\mathcal D'(\omega)$'.}
\]
We establish now in Section~\ref{S:distrbroad} a characterization more related to the Lagrangian formulation: we prove that one can reduce the PDE
\[
\nabla^{\phi}\phi=w
\]
along any integral line of the vector fields $\nabla^\phi_i$, $i=2,\dots,n$, provided that one chooses suitably the $\mathrm L^{\infty}(\omega;\R^{2n-1})$ representative $\hat w$ of the distribution identified by $w$.
We first prove in Section~\ref{Ss:distrtoLagr}, as an introduction, the weaker statement that one can reduce the PDE to ODEs along a selected family of characteristic constituting a Lagrangian parameterisation.
As well, the converse holds: if the ODEs on characteristics are satisfied one has a distributional solution to the PDE and the sources of the two formulations can be identified; this, as recalled in Lemma~\ref{L:lipDistr}, is known from~\cite{pinamonti} and we prove it differently in Section~\ref{Ss:lagrdistr} below.

The conclusion of this last section will be the following.

\begin{corollary}
The various notions of continuous solutions we have considered are equivalent.
\end{corollary}

\subsection{Distributional solutions are Lagrangian solutions}
\label{Ss:distrtoLagr}
The present section is an introduction to the next one, which proves the stronger statement that distributional solutions are broad solutions.
Being technically simpler, this section provides a guideline for some ideas implemented next.
One can in particular notice that the proof concerns only ODEs.

In order to avoid technicalities we focus on $n=2$ and on $\omega=\tilde\omega=[0,1]^{2}$. We refer to Section~\ref{S:distrbroad} for the stronger statement with the universal source term $\hat w$.
The function~$\bar w$ below relative to the Lagrangian formulation has not yet been related with the RHS of~\eqref{E:balance}: their identification will come from Theorem~\ref{T:converseDafermos}.

\begin{lemma}
\label{L:gfunztx}
Let $\phi$ be a continuous function. 
Consider a Lagrangian parameterisation $(\tilde\omega,\chi(s,\tau))$ 
and assume that $\phi(s, \chi(s,\tau))$ is Lipschitz in $s$ for all $\tau$. Then there exists a Borel function $\bar w:\omega\to\R$ such that for all $\tau$
\[
\partial_{s}\phi(s, \chi (s,\tau))= \partial_{ss} \chi (s,\tau) = \bar w(s, \chi (s,\tau)) 
\qquad
\text{for $\Ll^{1}$-a.e.~$s$.}
\]
\end{lemma}

\begin{corollary}
\label{C:lipordistrlagr}
If a function $\phi:\omega\subset\W\to \R$ is
\begin{itemize}
\item[-] either an intrinsic Lipschitz continuous function
\item[-] or a continuous distributional solutions $\phi$ to the balance law~\eqref{E:balance}
\end{itemize}
then it is a Lagrangian solution to the equation $\Wf\f=w$.
\end{corollary}

\begin{proof}[Proof of the corollary]
The existence of a Lagrangian parameterisation has been proved in Lemma~\ref{L:globchi}.
As recalled just after the Definition~\ref{D:deflip}, an intrinsic Lipschitz continuous function is continuous and Lipschitz along characteristics.
Also continuous distributional solutions to the balance law~\eqref{E:balance} are Lipschitz continuous along characteristics by Lemma~\ref{L:DafLipschitz}, following~\cite{Dafermos}). 
We have therefore the thesis directly by applying Lemma~\ref{L:gfunztx}.
\end{proof}

\begin{proof}[Proof of Lemma~\ref{L:gfunztx}]

We remind the notation
\[
	\Upsilon{}:\tilde\omega \to\omega
		\qquad
	\Upsilon{}(s,\tau):=(s,\chi (s,\tau)).
\]

{\sf Second derivative in $\tilde\omega$.}
By assumption $\partial_{s}\chi (s,\tau)=\phi(s,\chi (s,\tau))$ is Lipschitz in $s$.
Being $\phi$ continuous, one can see that the subset $B\subset\tilde\omega $ of those $(s,\tau)$ where $\chi (s,\tau) $ is twice $s$-differentiable is $F_{\sigma\delta}$, and $\partial_{ss} \chi (s,\tau)$ is a Borel function on it. Moreover, by Rademacher's theorem the $\tau$-sections of $B$ have full measure and therefore by Tonelli theorem also $B$ has full measure.
However, the function $\partial_{ss} \chi (s,\tau)$ is defined on $\tilde\omega $, while we are looking for a function defined on $\omega$.

{\sf A preliminary comment.}
In order to check that $\Upsilon{}$ lifts $\partial_{ss} \chi  $ to a map $\bar w$ a.e.~defined on $\omega$, which would provide our thesis, it would be natural to show that
\begin{itemize}
\item $\Upsilon{}(B)$ is a Lebesgue measurable subset of $\omega$ with full measure;
\item $\partial_{ss} \chi (s,\tau) $ is constant on the level set of $\Upsilon{}$ intersected with $B$.
\end{itemize}
Since $\chi $ is not Lipschitz, we do not manage to prove at this point that $\Upsilon{}(B) $ has full measure. We instead assign a specific value, $0$, to the function out of $\Upsilon{}(B) $. This choice of the extension does not affect our claim.
We notice moreover that the second point is true in that strong form, but we show that the points of the level set of $\Upsilon{}$ corresponding to more values of $\partial_{ss}\chi$ are not relevant. 

{\sf Analysis of $\Upsilon{}(B)$ and partial inverse of $\Upsilon{}$.}
The proof of the Borel measurability of $\Upsilon{}(B)$ requires some technicality: we apply a theorem due to Srivastava (\cite{Sri}, Theorem 5.9.2) deriving that there exists a Borel restriction $\chi $ which is one-to-one to $\Upsilon{}(B)$; then Theorem 4.12.4 in~\cite{Sri}, due to Lusin, would provide the thesis.
In order to apply the first theorem, we partition $\tilde\omega $ into the level sets of $\Upsilon{}$, which are $G_{\delta}$.
We need also to observe that $\Upsilon{}^{-1}(\Upsilon{})(O)$ is Borel for each open set $O$. For simplicity, consider the case when $\chi $ is already a full parameterisation and thus it is continuous. Every open set $O$ is $\sigma$-compact: thus by continuity $\Upsilon{}(O)$ is $\sigma$-compact, and finally $\Upsilon{}^{-1}(\Upsilon{})(O)$ is $\sigma$-compact. Therefore by Srivastava's theorem there is a Borel cross section $S$ for the partition: $\Upsilon{}$ restricted to $S\cap B$ is Borel, injective and onto $\Upsilon{}(B)$. Being a Borel image by a one-to-one map, Lusin's theorem asserts on one hand that that $\Upsilon{}(B)$ is Borel, and moreover that this restriction has a Borel inverse
\[
\Xi:\omega\to\tilde\omega .
\]

Maybe there is a more elementary argument which allows to approximate $B$ with a $\sigma$-compact subset $B_{K}$ whose $\tau$-section have full measure. Then one could as well work with $B_{K}$ instead of $B$ avoiding measurability difficulties, even without investigating the size of its image.

{\sf Analysis of $\partial_{ss} \chi (s,\tau) $ for $\tau\in C$.}
We can define $\bar w$ as
\[
\bar w(s,t)=
\begin{cases}
\partial_{ss} \chi (\Xi(s,\tau))
&(s,t)\in \Upsilon{}( B),
\\
0
& \text{otherwise.}
\end{cases}
\]
This map satisfies the claim by the next analysis of the set where $(\partial_{ss}\chi)\circ \Upsilon{}^{-1}$ is multivalued.

{\sf Analysis of multivalued points of $(\partial_{ss}  \chi) \circ\Upsilon{}^{-1} $.}
The set of points $(s,\tau)\in B$ such that there exists $(\bar s, \bar \tau)\in B$ satisfying $\Upsilon{}(s,\tau)=\Upsilon{}(\bar s, \bar \tau)$ and $\partial_{ss} \chi(s,\tau)\neq \partial_{ss} \chi(\bar s,\bar \tau)$ is Borel: just see it as
\[
M:=\Upsilon{}^{-1}\circ\Upsilon{}\big(\big\{(s,\tau)\in B: \qquad|\partial_{ss} \chi(s,\tau)-\partial_{ss} \chi(\Xi(\Upsilon{}(s,\tau))|>0\big\}\big).
\]
The value of $\bar w$ on this set is not at all relevant: we show below that for any $\tau$ the intersection of $\Upsilon{}(M)$ with the characteristics curve $\{\chi(s, \tau)\}_{s\in[0,1]}$ is at most countable.
\\
Indeed, fix a curve $\gamma(s)=\chi(s,\tau)$, for simplicity fix $ \gamma(s)\equiv 0$.
Suppose $\gamma'(s)=\chi(s,\tau')$ intersects $\gamma(s)$ at $s=\bar s$ with $\ddot\gamma'(\bar s)=\partial_{ss} \chi( \bar s,\tau)>1/m$, $m\in\N$ fixed.
By Taylor's expansion, there exists a neighborhood $(\bar s -\delta, \bar s +\delta)$ where
\[
\frac{(s-\bar s)}{3m}  \geq \gamma'(s) \geq \frac{2(s-\bar s)}{3m}.
\]
There can be at most $1/\delta$ such values $\bar s$ of $s$: otherwise by the last inequality two characteristics $\gamma'_{1}, \gamma'_{2}$ relative to two $\bar s_{1}<\bar s_{2}$ with $|s_{2}-s_{1}|\leq \delta$ would intersect at $\xi\in (s_{1},s_{2})$ and cross each other.
Taking the union for a sequence $\delta_{i}\downarrow 0$, and then the union over $m\in\N$, we get the thesis.
\end{proof}

\subsection{Distributional solutions are broad solutions}
\label{S:distrbroad}

Below we show Theorem~\ref{T:univselection}: there exists a Borel function $\hat w(y,t)$ such that every curve $\gamma(s)$ satisfying the ODE with continuous coefficient $\dot\gamma(s)=\phi(s,\gamma(s))$ has time derivative Lipschitz, and it has second derivative precisely $\hat w(s,\gamma(s))$ for a.e.~$s$.
The remarkable fact is that $\hat w(y,t)$ could be defined independently of any set of characteristic curves.
This is thus different, and stronger, from what we already proved, which is that there exists a Lagrangian parameterisation $\chi$ and an associated function $w_\chi$ satisfying $\partial^{2}_{s}\chi_{}(s,\tau)= w_\chi(s,\chi_{}(s,\tau))$.
The new point is indeed that $\hat w$ is a universal representative for the source term.

Due to its nature this proof works for any $n$ with no further complication. The only difference is that $\omega$ will be a subset of $\R^{2n}$ instead of $\R^{2}$. We write it with $n=1$ only for notational convenience. In particular we generalize here the previous Lemma~\ref{L:gfunztx}.

Since the argument is more intuitive, we mention first how to construct such a Souslin function $\hat w(y,t)$. We proceed then with the Borel construction because it gives a better result.

\subsubsection{Souslin selection}
\label{Sss:souslSel}
The first step is to define pointwise, but in a measurable way, a function $\hat w(y,t)$ such that $t$ is a Lebesgue point for the second derivative of a curve $s\mapsto\gamma(s)$ with $\gamma(y)=t$ and satisfying the ODE, whenever there exists one. Therefore one applies Von Neuman's selection theorem (Section 5.5 in~\cite{Sri}, from \cite{Dub}) to the subset of
\[
\omega\times \mathrm C^{1}\times\R
\supset
\mathcal G
\ni
(y,t,\gamma,\zeta)
\]
defined by
\[
\mathcal G^{}=
\bigg\{
\gamma(y)=t,
\  \dot\gamma =\phi\circ (\Id\otimes\gamma),
\ |\dot\gamma(r)-\dot\gamma (s)|\leq \norm{w}|r-s|,
\ \zeta=\lim_{\sigma\downarrow 0}\frac{1}{\pm \sigma}\int_{y}^{y\pm\sigma} \ddot\gamma(s)ds
\bigg\}.
\]
\begin{lemma}
$\mathcal G$ is Borel. It has full measure projection on $(y,t)\in \omega $.
\end{lemma}
\begin{proof}
{\sf Components $(y,t,\gamma)$.}
The subset 
\begin{equation}
\label{E:C}
(y,t,\gamma,\zeta)\ni C\subset \omega\times \mathrm C^{1}(\R)\times\R
\end{equation}
identified by the constraints
\[
\gamma(s)=x
\qquad
\dot\gamma =u\circ (\Id\otimes\gamma)
\qquad
|\dot\gamma(\tau)-\dot\gamma (\sigma)|\leq \norm{w}|\tau-\sigma|
\]
is closed.
By Lemma~\ref{L:DafLipschitz}, moreover, its projection on $(y,t)$ is all $\R^{+}\times\R$.
By Rademacher theorem, we have moreover seen in the same lemma also that the projection of $\mathcal G$ on $(y,t)$ has full measure.

{\sf Component $\zeta$, discretization.}
In order to establish the existence (and the value) of the limit for the second derivative of $\gamma$ at $y$, it suffices to consider e.g.~the sequence
\[
h_{n+1}=h_{n}-h_{n}^{2},
\qquad
h_{1}=1/2.
\]
Indeed, then for $h\in(h_{n+1},h_{n}]$, for example at $y=0$
\[
\bigg|
\frac{1}{h}\int_{0}^{h}\ddot \gamma - \frac{1}{h_{n}}\int_{0}^{h_{n}}\ddot \gamma
\bigg|
=
\bigg|
\left(\frac{1}{h}-\frac{1}{h_{n}}\right)\int_{0}^{h}\ddot \gamma 
	- \frac{1}{h_{n}}\int_{h}^{h_{n}}\ddot \gamma
\bigg|
\leq
2\norm{w}\frac{h_{n}-h}{ h_{n}} .
\]
By construction however
\[
|h_{n}-h|
\leq
|h_{n}-h_{n+1}|
=
h_{n}^{2},
\]
yielding that the existence of the limit along $\{h_{n}\}_{n}$ implies the existence of the limit for any $h\downarrow0$.
Notice that this would not hold choosing a generic $\tilde h_n\downarrow 0$ instead of $\{h_n\}_n$.

{\sf Measurability of $\mathcal G$.}
The further constraint in $\gamma$ can be written as
\[
\forall k\ \exists n\ \forall \bar n\geq n:
\qquad
\left| \zeta-\frac{1}{\pm h_{\bar n}}\int_{y}^{y\pm h_{\bar n}} \ddot\gamma(s)ds\right| \leq 2^{-k}.
\]
Therefore, we are considering the following subset of $C$:
\[
\mathcal G
=
C\cap \bigcap_{k\in\N} \bigcup_{n\in\N} \bigcap_{\bar n\geq n} 
	\left\{
	\left| \zeta-\frac{\dot\gamma(y\pm h_{\bar n})-\dot\gamma(y)}{\pm h_{\bar n}}\right| \leq 2^{-k}
	\right\}.
\]
Since the set within brackets is closed, $\mathcal G$ is Borel.
\end{proof}

This allows to define an $\mathcal A$-selection $(y,t)\mapsto\gamma_{(y,t)}$, which by construction is a measurable function assigning to every point $(y,t)$ an integral curve $\gamma_{(y,t)}$ for the ODE $\dot\gamma_{(y,t)}(s)=\phi(s, \gamma_{(y,t)}(s))$, whenever there exists a Lipschitz one having $y$ as a Lebesgue point for the right, left and total second derivative.
As well, one can define the Souslin function
\[
(y,t)\mapsto \hat w(y,t):=\ddot \gamma_{(y,t)}(t)=w_{(y,t)}.
\]
The importance of this selection is due to the following theorem.

\begin{theorem}
\label{T:univselection}
Let $\phi$ be a continuous function which is uniformly Lipschitz along characteristics curves $\gamma$: $\dot\gamma(s)=\phi(s, \gamma(s))$. 
Then \underline{for every} curve $\gamma$ satisfying the ODE $\dot\gamma(s)=\phi(s, \gamma(s))$, one has
\[
\frac{d}{dt}\phi(s,\gamma(s))=\hat w (s, \gamma(s)) 
\qquad
\text{at points of differentiability of $s\mapsto\phi(s,\gamma(s))$.}
\]
\end{theorem}

\begin{corollary}
\label{C:distrbroad}
A continuous distributional solution $\phi$ to $\gf\f=w $ is also a broad solution.
\end{corollary}

\subsubsection{Borel selection}
\label{Sss:borelSel}
Before proving the theorem, for the sake of completeness we show that one can define as well a Borel fucntion, that we still denote as $\hat w(y,t)=w_{y,t}$, for which Theorem~\ref{T:univselection} still holds. This requires a bit more work than the previous argument, and it is conceptually a little more involved: we do not associate immediately to each point (where it is possible) an eligible curve and its second derivative, but something which must be close to it.
We will find then with the proof of Theorem~\ref{T:univselection} that we end up basically with the same selection.

\begin{lemma}
\label{L:BorelApprg}
For every $\varepsilon>0$, there is a Borel function defined on the $(y,t)$-projection $\mathcal D$ of $\mathcal G$
\[
(y,t)\mapsto (\gamma_{\varepsilon,y,t},w_{\varepsilon,y,t})
\]
such that $(y,t,\gamma_{\varepsilon,y,t},w_{\varepsilon,y,t}) \in C$ of~\eqref{E:C} and that for $|h|$ sufficiently small
\[
\left| w_{\varepsilon,y,t} - \frac{1}{h} \int_{y}^{y+h}\ddot \gamma_{\varepsilon,y,t}\right|<\varepsilon .
\]
\end{lemma}

\begin{definition}
We define the Borel representative of $w$ as the function
\[
\hat w(y,t)=w_{y,t}=\chi_{\{(y,t)\in \mathcal D\}}\liminf_{\varepsilon\downarrow 0} w_{\varepsilon,y,t} .
\]
\end{definition}

\begin{proof}[Proof of Lemma~\ref{L:BorelApprg}]
We apply Arsenin-Kunugui selection theorem (\cite{Kun}, or Th.~5.12.1 of~\cite{Sri}) to the set
\begin{equation}
\label{E:accessorio2}
\bigcup_{n\in\N} \bigcap_{m>n}
\left\{
(y,t,\gamma_{},w) \in C:
\ \left| w - \frac{\dot \gamma(y+h_{m})-\dot\gamma(y)}{h_{m}}\right|\leq \varepsilon
\right\},
\end{equation}
where $C$ was defined in~\eqref{E:C} and $\{h_{n}\}_{n\in\N}$ immediately below that, in the same proof.
More precisely, $C$ is closed and
\[
\left\{
(\gamma,\zeta):\ 
\left| w - \frac{\dot \gamma(y+h)-\dot\gamma(y)}{h}\right|\leq\varepsilon .
\right\}
\]
is also closed: therefore each $(y,t)$-section of~\eqref{E:accessorio2} is $\sigma$-compact.
Then the hypothesis of the theorem are satisfied: it assures that the projection of~\eqref{E:accessorio2} on the first factor $\omega\ni(y,t)$ is Borel and there exists a Borel section of~\eqref{E:accessorio2} defied on it, which is the function in our statement.

Notice that the domain fo this function, containing $\mathcal{D}$, has full Lebesgue measure.
\end{proof}

\subsubsection{Proof of Theorem~\ref{T:univselection}}
\label{Sss:proofSel}
We provide now the proof of Theorem~\ref{T:univselection} with $w_{y,t}$ either the Borel or the Souslin one: we consider any characteristic $\bar \gamma$ for the balance law and we prove that for almost every $y$ its second derivative is precisely $w_{y,\bar\gamma(y)}$. We remind that $\dot{\bar\gamma}(y)$ is Lipschitz (Lemma~\ref{L:DafLipschitz}).

{\sf Step 1, countable decomposition.} The set of $y$ Lebesgue point of $\ddot{\bar\gamma}(y)$ with value different from $w_{t,\bar\gamma(y)}$ can be reduced to
\begin{align*}
&\bigcup_{\varepsilon\downarrow 0}
\left\{
s:\ 
\left| w_{y,\bar\gamma(y)} - \ddot {\bar\gamma}(y)\right|\geq\varepsilon
\right\}
\\
&\quad\subset
\bigcup_{\varepsilon\downarrow 0}
\bigcup_{n\in\N}
\left\{
y:\ \forall\sigma<2^{-n} \quad
\left|w_{y,\bar\gamma(y)} - \frac{1}{\sigma} \int_{y}^{y+ \sigma}\ddot {\bar\gamma}\right|\geq \varepsilon
, \quad 
\left|w_{y,\bar\gamma(y)} -
\frac{1}{\sigma} \int_{y-\sigma}^{y}\ddot {\bar\gamma} \right|\geq \varepsilon
\right\} 
\\
&\quad=
\bigcup_{\tilde \varepsilon<\varepsilon\downarrow 0}
\bigcup_{n\in\N}
\left\{
y:\ \forall \sigma <2^{-n} \quad 
\left| w_{y,\bar\gamma(y)} - 
\frac{1}{\pm \sigma} \int_{y}^{y\pm \sigma}\ddot {\bar\gamma}\right|\geq 3\varepsilon
, \quad
	\left| w_{t, \bar\gamma(s)} 
		- \frac{1}{\pm \sigma} \int_{y}^{y\pm \sigma}\ddot \gamma_{\tilde\varepsilon,t, \bar\gamma(s)}\right|
		<\varepsilon
\right\} 
.
\end{align*}
If one is considering the Souslin selection, clearly there is the simplification
$\gamma_{\tilde\varepsilon,y,t}=\gamma_{y,t}$.

{\sf Step 2, reduction argument.}
We prove that the set
\begin{equation}
\label{E:accessorio}
\left\{
y:\ \forall\sigma<2^{-n} \quad 
w_{y,\bar\gamma(y)} > 
\frac{1}{\pm \sigma} \int_{y}^{y\pm \sigma}\ddot {\bar\gamma}+3\varepsilon
, \quad
	\left| w_{y, \bar\gamma(y)} 
		- \frac{1}{\sigma} \int_{y}^{y+ \sigma}\ddot \gamma_{\tilde\varepsilon,y, \bar\gamma(y)}\right|
		<\varepsilon
\right\} 
\end{equation}
cannot contain points $y_{1}, y_{2}$ with $|y_{1}-y_{2}|\leq 2^{-n}$.
Then the thesis will follow: by the previous step the set of $y$ where the second derivative of $\bar\gamma(y)$ exists and it is different from $w_{t,\bar\gamma(y)}$ will be countable. Therefore the second derivative of $\bar\gamma(y)$ will be almost everywhere precisely $w_{t,\bar\gamma(y)} $.

{\sf Step 3: Analysis of the single sets.}
By contradiction, assume that~\eqref{E:accessorio} contains two such points, for example $y_{1}=0$, $y_{2}=\sigma$.
By definition of the set of points we are considering, the two selected curves through $(0,\bar \gamma(0))$, $(\sigma,\bar\gamma(\sigma))$,
\[
\gamma_{0}:=\gamma_{\tilde\varepsilon,0, \bar\gamma(0)},
\qquad
\gamma_{\sigma}:=\gamma_{\tilde\varepsilon, \sigma, \bar\gamma(\sigma)}
\]
must intersect in the time interval $[0, \sigma]$, say at time $\sigma'$.
Since they satisfy the ODE for characteristics, where they intersect they have the same derivative. Being all of them Lipschitz, we have then
\[
\dot{\bar\gamma}(0) + \int_{0}^{\sigma'}\ddot \gamma_{0}
= \dot\gamma_{\sigma}(\sigma')
=
\dot{\bar\gamma}(\sigma) - \int_{\sigma'}^{\sigma}\ddot \gamma_{\sigma}.
\]
Comparing the LHS and the RHS, one arrives to
\[
\dot{\bar\gamma}(\sigma) -\dot{\bar\gamma}(0) 
=
\int_{0}^{\sigma'}\ddot \gamma_{0} + \int_{\sigma'}^{\sigma}\ddot \gamma_{\sigma}
.
\]
However, since the times $0, \sigma $ belong by construction to the set~\eqref{E:accessorio} one has
\[
\int_{0}^{\sigma'}\ddot \gamma_{0} + \int_{\sigma'}^{\sigma}\ddot \gamma_{\sigma}
>
w_{0,\bar\gamma(0)}\sigma' + w_{\sigma,\bar\gamma(\sigma)} (\sigma-\sigma') - 2\varepsilon
>
\int_{0}^{\sigma} \ddot{\bar\gamma} + \varepsilon  
= \dot{\bar\gamma}(\sigma) -\dot{\bar\gamma}(0) +\varepsilon
\]
reaching a contradiction.
$\hfill\qed$

\subsection{Lagrangian solutions are distributional solutions}
\label{Ss:lagrdistr}
In this section we prove, without passing through the implicit function theorem, that if a continuous function $\phi$ satisfies the Lagrangian formulation~\ref{D:Lagrsol} of the balance law~\eqref{E:balance}, then in the Eulerian variables $\phi$ solves indeed the balance law in distributional sense. 
We instead rely on a mollification procedure in the Legrangian variables.
See also \cite{pinamonti}, where a different, pointwise approximation of the distributional solution is provided, starting from a broad* solution.
This basically shows the converse of Dafermos' statement in~\cite{Dafermos}.

We notice that the fact that if a continuous $u$ solves the Lagrangian formulation, then it is not difficult as we see in the next theorem proving that it solves a balance law in distributional sense. The identification of the Lagrangian and distributional sources however is not trivial: we indeed use the converse implication of Section~\ref{S:distrbroad} and Rademacher theorem.

We already motivated why focusing on the case $n=1$, whereas the equation reduces to
\[
\tag{\ref{E:balance}}
\phi_{y}(y,t)+ \left[\frac{\phi^{2}(y,t)}{2}\right]_{t}=w(y,t).
\]
We give generalizations to the case $n\geq2$ in Remark \ref{R:remarkdimnsection6}

\begin{theorem}
\label{T:converseDafermos}
Every Lagrangian solution to~\eqref{E:balance} is also a distributional solution, with source term given by the distribution identified by the function $\bar w$ in the Lagrangian formulation. 
\end{theorem}

\begin{corollary}\label{C:broaddistrib}
Any continuous broad solution $\phi$ to~\eqref{E:balance} is also a distributional solution.
\end{corollary}

%

%
\begin{proof}[Proof of Theorem~\ref{T:converseDafermos}]
Let $\chi(s,\tau)$ be a Lagrangian parameterisation associated to $\phi$, and $w:\omega\to\R$ such that
\[
\phi(s,\chi(s,\tau))-\phi(0,\chi(0, \tau))=\int_{0}^{s} w (r, \chi(r, \tau)) dr .
\]
We prove then that $\phi$ is a distributional solution of the balance equation
\[
\phi_{y}(y,t)+ \left[\frac{\phi^{2}(y,t)}{2}\right]_{t}=w(y,t).
\]

{\sf Smoothing of $\phi, \chi$ in the $\tau$-variable.}
Consider a suitable convolution kernel $\rho_{\varepsilon}(\tau)$ and define the $\tau$-regularized function $\chi^{\varepsilon}(s,\tau)$ given by
\begin{align*}
&\chi^{\varepsilon}(s,\tau)=\chi(s,\tau)* \rho_{\varepsilon}(\tau) =\int_{\R} \chi(s,\xi)\rho_{\varepsilon}(\tau-\xi) d\xi=\int_{-\varepsilon}^{\varepsilon} \chi(s,\xi)\rho_{\varepsilon}(\tau-\xi) d\xi.
\end{align*}
This function $\chi^{\varepsilon}$ is smooth but possibly still not injective. However, by the monotonicity of $\chi$ one has that if $\chi^{\varepsilon}(s,\tau_{1})=\chi^{\varepsilon}(s,\tau_{2})$ then $\chi(s,\tau)$ is constant for $\tau\in[\tau_{1}-\varepsilon,\tau_{2}+\varepsilon]$. As a consequence, $\phi(s,\chi(s,\tau))$ is in turn constant for $\tau\in[\tau_{1}-\varepsilon,\tau_{2}+\varepsilon]$.

The above observation allows to define consequently the approximation $\phi^{\varepsilon}(s,t)$ by
\[
\phi^{\varepsilon}(s,\chi^{\varepsilon}(s,\tau))=\partial_{s}\chi^{\varepsilon}(s,\tau)=\partial_{s}\chi(s,\tau)* \rho_{\varepsilon}(\tau)=\phi(s,\chi(s,\tau))* \rho_{\varepsilon}(\tau).
\]
Notice that $\phi^{\varepsilon}$ is well defined and continuous in the $(s,t)$-variables. The measurable subset of points where it is not differentiable is possibly non-empty as a consequence of the non-injectivity of $\chi^{\varepsilon}$, but it is at most countable on every $s$-section. $\phi^{\varepsilon}$ is $C^{\infty}$ on the remaining full measure set. 

Since both $\chi$ and $\phi$ are continuous, the above relations immediately imply the local uniform convergence of the regularized functions:
\begin{align*}
&\chi^{\varepsilon}(s,\tau)
\rightrightarrows  \chi(s,\tau)
&&
\phi^{\varepsilon}(s,\chi^{\varepsilon}(s,\tau)) \rightrightarrows  \phi(s,\chi(s,\tau)) .
\end{align*}

{\sf Convergence in the $(s,t)$-variables of $\phi^{\varepsilon}$.}
Notice that $\partial_{\tau}\chi^{\varepsilon}(s,\tau)dsd\tau$ is converging (as a measure) to $\partial_{\tau}\chi^{}(s,\tau)$, and that their variation is unifromly bounded.
The above procedure not only defines correctly the continuous functions $\phi^{\varepsilon}(s,t)$, but it allows to establish their convergence in $\mathrm L^{1}$: indeed
\begin{align*}
\int |\phi^{\varepsilon}(s,t)-\phi(s,t)| dy 
&= 
\int | \phi^{\varepsilon}(s, \chi^{\varepsilon}(s,\tau))-\phi(s,\chi^{\varepsilon}(s,\tau))| \partial_{\tau}\chi^{\varepsilon}(s,\tau) d\tau 
\\
&\leq 
\int | \phi^{\varepsilon}(s, \chi^{\varepsilon}(s,\tau))-\phi(s,\chi^{}(s,\tau))| \partial_{\tau}\chi^{\varepsilon}(s,\tau) d\tau 
\\
&\qquad+
\int |\phi(s,\chi^{}(s,\tau))-\phi(s,\chi^{\varepsilon}(s,\tau))| \partial_{\tau}\chi^{\varepsilon}(s,\tau) d\tau .
\end{align*}
The first factor in both the integrals is uniformly convergent to zero, while $\partial_{\tau}\chi^{\varepsilon}(s,\tau) dy$ converges to the measure $\partial_{\tau}\chi^{}(s,\tau) $.
The convergence of $(\phi^{\varepsilon})^{2}/2$ is then straightforward.

{\sf Approximation of the source.}
As $\phi^{\varepsilon}$ is a smooth function a.e., one can define an approximate source $w^{\varepsilon}(s,t)$ by
\begin{equation}
\label{E:PDEapprox}
\dede{y}\phi^{\varepsilon}(y,t)+\dede{t}\frac{(\phi^{\varepsilon})^{2}}{2}(y,t)=w^{\varepsilon}(y,t)
\qquad
\dede{s}\chi^{\varepsilon}(s,\tau)=\phi^{\varepsilon}(s,\tau)
.
\end{equation}
The above relation, by the pointwise smoothness is immediately equivalent to
\[
w^{\varepsilon}(s,\chi^{\varepsilon}(s,\tau))
=\partial_{ss}\chi^{\varepsilon}(s,\tau)
=\dede{s}\phi^{\varepsilon}(s,\chi^{\varepsilon}(s,\tau))
 .
\]

Since we started from a Lagrangian parameterisation, the further regularity in the $s$ variable
\[
\partial_{ss}\chi(s,\tau)=\dede{s} \phi(s,\chi(s,\tau))= \bar w(s,\chi(s,\tau))
\]
for the relative pointwise representative $\bar w$ implies the relation
\[
w^{\varepsilon}(s,\chi^{\varepsilon}(s,\tau))
=\partial_{ss}\chi^{\varepsilon}(s,\tau) 
=\partial_{ss}\chi(s,\tau)*\rho_{\varepsilon}(\tau)
= \bar w(s,\chi(s,\tau))*\rho_{\varepsilon}(\tau)
.
\]
In particular, the sources $w^{\varepsilon}$ are uniformly bounded by the $\mathrm L^{\infty}$ bound for $\bar w$.
Moreover, for each $t$ fixed $w^{\varepsilon}(s,\chi^{\varepsilon}(s,\tau))$ converges in all $\mathrm L^{p}(dt)$ to $\bar w(s,\chi(s,\tau))$, and thus in $\mathrm L^{p}(dydt)$; the convergence is clearly uniform when $\bar w$ is continuous.

The LHS of equation~\eqref{E:PDEapprox} passes to the weak limit by the $\mathrm L^{1}$-convergence of $\phi^{\varepsilon}(y,t)$ to $\phi(y,t)$ established above. The same holds as well for the RHS, since $w^{\varepsilon}(y,t)dydt$ converge in $w^{*}-L^{\infty}$ to a function $\bar{\bar w}(y,t)dydt$.
Thus
\[
\phi^{}_{y}(y,t)+\bigg[\frac{\phi^{2}(y,t)}{2}\bigg]_{t}= \bar{\bar w}^{}(y,t).
\]
The function $\phi$ as we have seen is precisely the one in the Lagrangian parameterisation.
Even if $w^{\varepsilon}(s,\chi^{\varepsilon}(s,\tau))$ converges in all $\mathrm L^{p}(dt)$ to $\bar w(s,\chi(s,\tau))$, it is not instead trivial that above $\bar{\bar w}={\bar w}$, Lebesgue almost everywhere, holds. We now explain why it is so.

We prove above in Section~\ref{S:distrbroad} that if
\[
\phi^{}_{y}(y,t)+\bigg[\frac{\phi^{2}(y,t)}{2}\bigg]_{t}= \bar{\bar w}^{}(y,t)
\]
holds then $\phi$ is also a broad solution: there is a point-wise representative $\hat{\bar{\bar w}}$ of $\bar{\bar w}$ such that the broad formulation holds with $\phi, \hat{\bar{\bar w}}$.
However, it is then also intrinsic Lipschitz continuous with intrinsic gradient $\hat{\bar{\bar w}}$: by Rademacher theorem the intrinsic gradient ${\hat w}$ is uniquely defined point-wise almost everywhere, and it by Lemma~\ref{L:lipDistr} it identifies the same distribution as $\bar{\bar w}^{}(y,t)$.
\end{proof}


\begin{remark}\label{R:remarkdimnsection6}{\rm
We finally remark that Theorem~\ref{T:converseDafermos} works immediately also in higher dimensions, because for (almost every) $v\in\R^{2(n-1)}$ the restriction to the plane $\omega_{v}$ of $\phi$ is still a Lagrangian solution, with source term the restriction $\tilde w\restriction_{\omega_{v}}$. For every test function $\varphi\in\ \mathrm C^{\infty}_{c}(\omega, \R)$ we have then by Fubini-Tonelli Theorem
\begin{align*}
\iint_{\omega} 
&\left[ \varphi_{y_{1}} \phi + \varphi_{t} \frac{\phi^{2}}{2} \right]
=
\iint_{\R^{2(n-1)}} dv\int_{\omega_{v}} 
dy_{1}dt
\left[ \varphi_{y_{1}} \phi + \varphi_{t} \frac{\phi^{2}}{2} \right]\restr{\omega_{v}}
\\
&\stackrel{\text{Th.~\ref{T:converseDafermos} }}{=}
-\iint_{\R^{2(n-1)}} dv\int_{\omega_{v}} 
dy_{1}dt
 \left[\tilde w\restr{\omega_{v}}\right]
=
-\iint_{\omega}\tilde w.
\end{align*}
Considering also the linear fields we gain the implication from \eqref{item:lagr} to \eqref{item:distr} in Theorem~\ref{T:othertheorem}, and more generally that a Lagrangian solution is also a distributional solution.}
\end{remark}


%
\appendix

\section{From partial to full Lagrangian parameterisations}
\label{S:extPar}

In the present section we deal with the issue of extending a partial Lagrangian parameterisation to a `full' one. We construct a function $\chi(s,\tau)$ satisfying the ODE~\eqref{E:generalODE}, both monotone and surjective in the $\tau$ variable, which extends a given one $\tilde\chi$.
This is the matter of Lemma~\ref{L:fillinODE}: we recall below how to extend a solution to an ODE with continuous coefficients, whose existence is a classical result.

The procedure can be first understood considering Example~\ref{E:fillholesNoGlobLip} below, illustrated in Figure~\ref{fig:parameterisation}. This deals with the simpler case of an $s$-independent $\phi$, but it has all the ingredients of the general construction of Lemma~\ref{L:fillinODE}.

Moreover, {Example~\ref{E:fillholesNoGlobLip} provides a counterexample for the following fact}: even if characteristics are $\mathrm C^{1}$ in $s$ with Lipschitz derivative, \emph{it is not possible in general to extend a partial, monotone Lagrangian parameterisation to a full one which is locally Lipschitz continuous}.

The reduction of the balance law along characteristics, which is Equation~\eqref{E:evolutionphi}, has been inserted in the text (Lemma~\ref{L:globchi} and Theorem~\ref{T:univselection}).
Here we just notice that it holds, with some $\bar w_{\chi}$ pointwise defined in $\omega$, for a particular Lagrangian parameterisation $(\tilde\omega, \chi)$.
The argument is $2$-dimensional.

\begin{lemma}
\label{L:fillinODE}
Any partial Lagrangian parameterisation can be extended to a full one.
\end{lemma}

\begin{proof}
Let $\tilde\chi{}(s,\tau)$ be a partial Lagrangian parameterisation.
Focus e.g.~the attention on $s,\tau\in[0,1]$ and $\tilde\chi(s,\tau)$ valued in $[0,1]$ and right continuous, the general case being similar.
We fix also $\lVert{\phi}\rVert_{\infty}\leq 1$.

We construct an extension $\tilde\chi'$ by a recursive procedure. For convenience, the induction index is given by couples $(h,n)$ with $n\in\N$ and $h=0,\dots,2^{n-1}-1$. The ordering is lexicographic, starting from the second variable: $(h_{1},n_{1})\leq (h_{2},n_{2})$ iff either $n_{1}< n_{2}$ or $n_{1}= n_{2}$ and $h_{1}\leq 
 h_{2}$.

The starting point is $\chi^{0}= \tilde \chi$ defined for $s\in[0,1]$, $\tau\in T_{0}=[0,1]$. Consider the dichotomous points $s^{h,n}=2^{-n}+2^{-n+1}h$, which go from $2^{-n}$ to $1-2^{-n}$ at step $2^{-n+1}$, associated to the indexes $(h,n)$ with $n\in\N$ and $h=0,\dots,2^{n-1}-1$. 

{\sf Induction step $(h,n)$, $n\geq 1$: General description.} 
Assume you have been given $\chi$ defined on $(s,\tau)\in [0,1]\times T $ by a previous step.
If at $s=s^{h,n}$ the map $\tau\mapsto\chi(s^{h,n},\tau)$ is not onto $[\chi(s^{h,n},0),\chi(s^{h,n},1)]$ we construct an extension $\chi^{h,n}$ such that 
\begin{itemize}
\item[-] $\tau\mapsto\chi^{h,n}(s^{h,n},\tau)$ is onto $[\chi(s^{h,n},0), \chi(s^{h,n},1)]$ for $h=0,\dots,2^{n-1}-1$;
\item[-] there exists a strictly increasing map $j^{h,n}$, with $\Ll^{1}(j^{h,n}(T))-\Ll^{1}(T)\leq 2^{1-2n}$, such that
	\[\chi^{h,n}(s,j^{h,n}(\tau))=\chi(s,\tau).\]
\end{itemize}
These properties of the new partial Lagrangian parameterisation will determine that we get at the end a limit which is a full parameterisation extending $\tilde\chi$.

{\sf Induction step $(h,n)$, $n\geq 1$: Change of parameter set.} 
Because of monotonicity the complementary of the image of $\tau\mapsto \chi(s^{h,n},\tau)$ is the at most countable union of disjoint intervals $\{I_{k}\}_{k}$, which correspond to the discontinuities $\{\tau^{}_{k}\}_{k}$ of this real valued map.
At those parameters $\{\tau_{k}\}_{k}$ the two characteristics $s\mapsto\chi^{}(s,\tau_{k}^{-})$ and $s\mapsto\chi^{}(s,\tau_{k}^{+})$ bifurcate, and at time $s^{h,n}$ their opening is an interval $I_{k}^{}$:
\[
\min I_{k}^{}=\chi^{}(s^{h,n},(\tau^{}_{k})^{-})
\qquad
\sup I_{k}^{}=\chi^{}(s^{h,n},\tau^{}_{k}).
\]

Define consequently the strictly increasing map opening each of those parameters $\tau_{k}$ into an interval proportional (with factor $1/2^{2n-1}$) to the hole $I_{k}$ that the relative characteristics leave at $s^{h,n}$:
\begin{align*}
&&T&&\to& & T_{h,n}=[0,j^{h,n}(1)]&&
\\
j^{h,n}&:& \tau&&\mapsto& & \tau+\Ll^{1}(\cup_{\tau_{k}^{}\leq \tau}I_{k}^{})/2^{2n-1} .&&
\end{align*}
The inequality $|T_{h,n}|-|T|\leq 2^{1-2n}$ holds because we are assuming that $\chi$ is valued in $[0,1]$, thus $\Ll^{1}(\cup_{k} I_{k})\leq 1$.
The set $T_{h,n}$ will be the new space of parameters, and the injection $j^{h,n}$ will bring from the old set $T$ to the new one.

{\sf Induction step $(h,n)$, $n\geq 1$: Extension of the parameterisation.} 
We just constructed a new parameter set $T_{h,n}$ and an immersion $j^{h,n}$ from the old one $T$.
In particular, the Lagrangian parameterisation is fixed on $j^{h,n}(T)$, but in order to conclude the induction step we need to define the new Lagrangian parameterisation on $T_{h,n}\setminus j^{h,n}(T)$.
We clearly need to respect also the monotonicity property, taking into account that part of the parameterisation is already fixed: the new characteristics that we are going to to define must not cross the old ones.

For each $z\in I_{k}^{}$, consider the $\mathrm C^{1}$ maximal curve through $(s^{h,n},z)$ defined at~\eqref{E:gammamM} until it touches either $s\mapsto\chi^{}(s,\tau_{k}^{+}) $ or $s\mapsto\chi^{}(s,\tau_{k}^{-}) $, in which case it goes on with the curve which has been touched. It is a little complicated to write formally, but it is just that. The times when this touching happens, if it ever happens, are
\begin{align*}
&s_{+}^{\pm}(s^{h,n},z)=\inf \Big\{ 1, \ \{s>s^{h,n}:\  \gamma^{(s^{h,n},z)}(s) = \chi^{}(s,\tau_{k}^{\pm})\} \Big\}
\\ 
&s_{-}^{\pm}(s^{h,n},z)=\sup \Big\{ 0,\ \{s>s^{h,n}:\  \gamma^{(s^{h,n},z)}(s) = \chi^{}(s,\tau_{k}^{\pm})\} \Big\}.
\end{align*}
With the notation that $\{s: a \leq s \leq b \}$ is empty if $b<a$, a possible right continuous extension is then give by
\[
\gamma(s;s^{h,n},z):=
\begin{cases}
\gamma^{(s^{h,n},z)}(s)
&\text{for $ s_{-}^{-}(s^{h,n},z) \bigvee s_{-}^{+}(s^{h,n},z)  \leq s \leq  s_{+}^{-}(s^{h,n},z) \bigwedge s_{+}^{+}(s^{h,n},z)$,}
\\
\chi^{}(s,\tau_{k}^{+})
&\text{for $  s_{+}^{+}(s^{h,n},z) \leq s \leq  s_{+}^{-}(s^{h,n},z)$ and $s_{-}^{-}(s^{h,n},z) \leq s \leq  s_{-}^{+}(s^{h,n},z)$}
\\
\chi^{}(s,\tau_{k}^{-})
&\text{otherwise.}
\end{cases}
\]

Define finally on $[0,1]\times T_{h,n}$ the Lagrangian parameterisation which coincides with the previous one on the image of the old parameter set, and which is extended as described above elsewhere:
\[
\chi^{h,n}(s, \sigma)=
\begin{cases}
\text{if $\sigma=j^{h,n}(\tau)$} 
&
\chi^{}(s,\tau) ,
\\
\text{if $\sigma\in \big[j^{h,n}(\tau_{k}^{-}), j^{h,n}(\tau_{k})\big)$, $z:=\chi^{}(s^{h,k},\tau_{k}^{+}) - 2^{2n-1} \big(j^{h,n}(\tau_{k}^{+})-\sigma\big)$}
&
\gamma(s;s^{h,n},z) .
\end{cases}
\]
For the second line, $j^{h,n}(\tau_{k}^{+})-\sigma$ is the length of the interval $\llbracket \sigma, j^{h,n}(\tau_{k}^{+})\rrbracket$, part of the ones added to the parameter set precisely at the $(h,n)$-th step. Rescaled by $2^{2n-1}$, it gives the length of the segment $\llbracket z, \chi^{}(s^{h,k},\tau_{ k}^{+})\rrbracket$: the point $(s^{h,n},z)$ is where we start for defining the characteristic $\gamma(s;s^{h,n},z)$ that we are inserting for extending the Lagrangian parameterisation.
Notice that also the surjectivity property at $s=s^{h,n}$ is satisfied.

{\sf Conclusion.} 
Let us first look at how much the domain $T_{0}$ grows in the extension process.
Since $2^{n-1}$ couples of indices have second variable $n$, then the total size of the intervals added by those couples alltogether is at most $2^{-n}$: thus, setting $T_{n}=\cup_{h} T_{h,n}$,
\[
T_{0}=[0,1],
\quad
T_{1}\subset [0,3/2],
\quad
\dots\ ,
\quad
T_{n}\subset [0,2-2^{-n}],
\quad
\dots\ .
\]
The maps $ j^{\bar h,\bar n}\circ\dots\circ j^{0,1}$ are
\begin{itemize}
\item[-] strictly increasing;
\item[-] valued in $[0,2-2^{-\bar n}]$;
\item[-] with $1$-Lipschitz inverses, provided that at discontinuity points of $ j^{\bar h,\bar n}$ one fills the graph.
\end{itemize}
By Ascoli-Arzel\`a theorem the inverses converge uniformly, to a monotone map which is the inverse of a right continuous function $j:[0,1]\to[0,2]$. By construction $j$ is strictly increasing.

By construction each $\chi^{\bar h,\bar n}$ is surjective at each $s=s^{h,n}$ with $(h,n)\leq (\bar h,\bar n)$.

The Lagrangian parameterisations $\chi^{\bar h,\bar n}$ and  $\chi^{h, n}$ have different domains for the second components, the space of parameters $T_{h,n}$, $T_{\bar h,\bar n}$. However, as seen in the previous step there are strictly monotone injections from each of them to the interval $\llbracket 0,2 \rrbracket$ given by $\circ_{\substack{(\ell,m)\geq ( h, n)}} j^{\ell,m}$ and $\circ_{\substack{(\ell,m)\geq (\bar h,\bar n)}} j^{\ell,m}$. Being strictly monotone, they are invertible if we fill the graphs at discontinuity points: we can compare these compositions, with the same second component in $\llbracket 0,2 \rrbracket$, and we find for them
\[
\lVert\chi^{\bar h,\bar n}-\chi^{h, n}\rVert_{\infty}\leq 2^{-n}.
\]
The sequence of these compositions therefore converges uniformly. 
One verifies that the limit is a monotone Lagrangian parameterisation $\chi$ which extends $\tilde \chi$, with injection map $j$.


Being surjective and monotone, each $\tau\to \chi(s^{h,m},\tau)$ is continuous. By the continuity in $s$ we deduce then surjectivity also at the remaining times: indeed if by absurd we had $\chi(\bar s, \bar\tau^{-})\neq \chi(\bar s, \bar\tau^{+}) $, we could not have $\chi(s^{h,m}, \bar\tau^{-})= \chi(s^{h,m}, \bar\tau^{+}) $ at $s^{h,m}$ arbitrarily close to $\bar s$.
\end{proof}

The following example introduces the extension of a Lagrangian parameterisation. It shows moreover that it is not possible in general to get a full one which is Lipschitz continuous, even though characteristics below are twice continuously differentiable.

\begin{example}
\label{E:fillholesNoGlobLip}{
Consider the very simple equation for $(z,t)$ in the rectangle $[z_{\infty},0]\times[0,1]$
\[
\left[\frac{\phi^{2}(z,t)}{2}\right]_{t}=w(z,t),
\qquad
\phi_{z}(z,t)=0.
\]
Being $\phi$ dependent on one variable, we change notation and we write $\phi(z,t)=\phi(z)$.
\\
We consider the partial Lagrangian parameterisation $\chi_{m}$ defined in Lemma~\ref{L:solODE}.
This would need to specify the function $\phi$: since the construction is involved, we specify it below and the reader can immediately visualise it in Figure~\ref{fig:parameterisation} (left side), where a family of integral curves $\dot\gamma(s)=\phi(\gamma(s))$ is drown.
\begin{figure}[htp!]
\includegraphics[width=.8\linewidth]{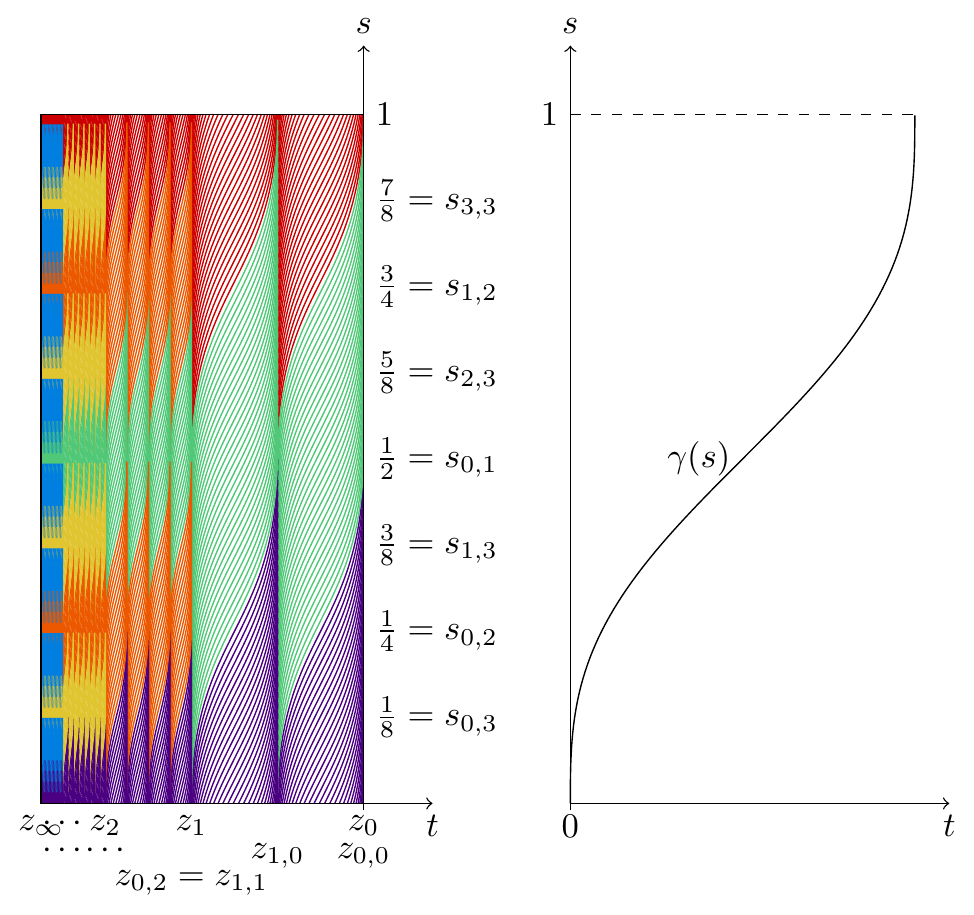}
\caption{A partial monotone Lagrangian parameterisation is extended to a `full' one. Different colors denote different steps in the extension, which are countably many. Each step corresponds to a dichotomous value of $s$: the relative extension of the parameterisation must cover the relative $s$-section. In this example the full parameterisation can not be locally Lipschitz.}
\label{fig:parameterisation}
\end{figure}

We now describe in details the construction.

{\sf Step 1: Building block (Figure~\ref{fig:parameterisation}, right side).}
Define first a smooth function $\gamma(s)$, for $s\in[0,1]$, which increases continuously from $0$ to $1$.
Let $\dot\gamma(s-1/2)$ be even, strictly increasing in the first half interval from $0$ to its maximum.
Let $\ddot \gamma $ vanish at $0;1/2;1$ and be positive in $[0,1/2]$.
For instance, consider
\[
\gamma(s)=s+1/(2\pi)\sin(2\pi s-\pi),
\qquad
\gamma'(s)=1+\cos(2\pi s-\pi).
\]

{\sf Step 2: Iteration step.}
We define now a first sequence of points  $z_{0}\geq z_{1} \geq \dots z_{i}\downarrow z_{\infty}$ and intermediate ones $z_{i}=z_{2^{i},i} \leq \dots \leq z_{0,i}=z_{i-1}$ where $\phi(z)$ and $w(z)$ will vanish.
The first ones are approximatively $0$, $-0.455$, $-0.635$, $-0.713$, $-0.748$, $-0.764$, $-0.772$, \dots , and each interval $[z_{i+1},z_{i}]$ is divided into $2^i$ equal subintervals for determining the points $z_{h,i}$.
For $i\in\N$, $h=0,\dots,2^{i}$
\begin{gather*}
z_{0}:=0,
\qquad
z_{i-1} - z_{i} = \frac{2^{-i}}{\ln (i+2)}
\qquad
{z_{\infty}:=-\sum_{j=1}^{\infty}\frac{2^{-j}}{\ln (j+2)}},
\\
z_{i}=-\sum_{j=1}^{i}\frac{2^{-j}}{\ln (j+2)}=z_{2^{i},i}=z_{0,i+1}
,
\qquad
z_{h,i}:=z_{i-1}-\frac{2^{-2i}h}{\ln (i+2)}
.
\end{gather*}

{\sf Step 3: Iteration.}
Now we define the functions $\phi(z)$, $w(z)$ on each subinterval $[z_{h+1,i},z_{h,i}]$. As a preliminary half-step consider the rescaled smooth functions $\gamma_{h,i}(s)$ given by
\[
\gamma_{h,i}(s)=z_{h+1,i}+\frac{\gamma(2^{i}s)}{2^{2i}\ln (i+2)},
\qquad
s\in [0,2^{-i}],
\]
Notice that each increases monotonically from $z_{h+1,i}$ to $z_{h,i}$, and the first two derivatives vanish at the endpoints.
In particular, we can associate to each $z\in[z_{\infty}, 0]$ a curve $\gamma_{h,i}(s)$ such that
\[
\gamma_{h,i}(s)=z.
\]
It will be unique out of the points $\{z_{h,i}\}_{h,i}$, while at these junctions there will be two such curves, with however vanishing first two derivatives.
Observing Figure~\ref{fig:parameterisation}, this curve $\gamma_{h,i}(s)$ is just the first segment of what will be part of $\chi_{m}$
\[
[0,2^{i}]\ni s \mapsto \gamma_{h,i}(s) \equiv \chi_{m}(s,z^{+}_{h+1,i}).
\]

Define then,
\[
\text{for $s$: }\gamma_{h,i}(s)=z,
\quad\quad
\phi(z)=\dot\gamma_{h,i}(s)
\quad \&\quad
w(z)=\ddot \gamma_{h,i}(s)
.
\]
These functions are clearly continuous out of the nodes $\{z_{h,i}\}_{h,i}$, and they vanish there, but the continuity at the limit point $z_{\infty} $ should be checked. It holds because $|\dot\gamma_{h,i}|\leq 2^{-i}/\ln (i+2)$ and $|\ddot\gamma_{h,i}|\leq 1/\log (i+2)$, which implies they vanish for $i\uparrow \infty$.

{\sf Step 4: non-Lipschitz Lagrangian Parameterisation.}
Having defined the function $\phi$, we have already defined the partial Lagrangian parameterisation $\chi_{m}$ of Lemma~\ref{L:solODE}.
We now extend it to a surjective one, but there is no way of having it Lipschitz continuous, as we compute now.
Different colors in Figure~\ref{fig:parameterisation} show different steps of the extension process.

1. Minimal characteristics starting from $z=0$ do not cover almost all the interval at $z=1$. Adding those starting at $z=1$, it remains to cover at $z=1/2$ open intervals of total length
\[
\sum_{i=1}^{\infty}
\sum_{h=0}^{2^{i}}
(z_{h,i}-z_{h+1,i})
=
\sum_{i=1}^{\infty}2^{i}\frac{2^{-2i}}{\ln (i+2)}=z_{0}-z_{\infty}.
\]
2. Including at the second step all those minimal characteristics which intersect the line $z=1/2$, similarly, one does not cover the whole line $z=1/4;3/4$: a length
\[
\sum_{i=2}^{\infty}\frac{2^{-i}}{\ln (i+2) }= z_1 - z_{\infty}
\]
remains to cover at both those two values of $z$.
\\
i. At the subsequent $i$-th step, at each of the $2^{i-1}$ values of
\[
s_{h,i}=2^{-i}+h2^{-i+1},
\qquad h=0,\dots,2^{i-1}-1
\]
one needs covering a length $\sum_{j=i}^{\infty}2^{-j}/\ln (j+2) $.
In the whole process, it must be covered a total length equal to
\begin{align*}
\sum_{j=1}^{\infty} &(1+\dots+2^{j-1})\frac{2^{-j}}{\ln (j+2)}
\\
&=
\sum_{j=1}^{\infty} (2^{j}-1)\frac{2^{-j}}{\ln (j+2)}
\\
&=\sum_{j=1}^{\infty} \frac{1}{\ln (j+2)}
-
\sum_{j=1}^{\infty} \frac{2^{-j}}{\ln (j+2)}
=
+\infty.
\end{align*}

Any monotone, Lagrangian parameterisation $\chi$ must map a disjoint family of real intervals $\{I_{h,i}\}_{h,i}$ with $\sum_{h,i} |I_{h,i}|<1$ to the intervals $\{[z_{\infty},z_{i+1}]\}_{i}$, respectively at the above valies $\{s_{h,i}\}_{h,i}$.
However, we have just computed that for all constants $C$
\begin{align*}
\infty 
=&
\sum_{h=0,\dots,2^{i-1}-1,\ i\in\N} |z_{i+1}-z_{\infty}|
\\
=&
\sum_{h=0,\dots,2^{i-1}-1,\ i\in\N} \chi(z_{h,i}, [I_{h,i}])
\\
>&
C\sum_{h=0,\dots,2^{i-1}-1,\ i\in\N}| I_{h,i}| ,
\end{align*}
preventing any Lipschitz regularity. Indeed, the map $\tau\mapsto\chi(s_{h,i},\tau)$ can be Lipschitz with some constant $C_i$, but $C_i$ must blow up as $i\to\infty$. At other values of $s$, this map is just $\mathrm W^{1,1}(\R)$.

Notice finally that the fact that the blow-up of the Lipschitz constant is caused here by a behaviour approaching the boundary $z=z_{\infty}$ is incidental. Indeed, one can extend the present construction also on $\{z< z_{\infty} \}$ for example reflecting the characteristics w.r.t.~the axis $\{z=z_{\infty}\}$.

One can as well construct examples where $\chi$ has a Cantor part (see~\cite{abc}).}
\end{example}


\end{document}